\documentclass[12pt]{amsart}
\usepackage{amsmath}
\usepackage{amsxtra}
\usepackage{amscd}
\usepackage{amsthm}
\usepackage{amsfonts}
\usepackage{amssymb}
\usepackage{eucal}
\usepackage{verbatim}
\usepackage[all]{xy}
\usepackage{mathrsfs}
\usepackage{color}
\definecolor{deepjunglegreen}{rgb}{0.0, 0.29, 0.29}
\definecolor{darkspringgreen}{rgb}{0.09, 0.45, 0.27}

\usepackage{stmaryrd}

\makeatletter
\usepackage{etex,etoolbox,needspace}
\pretocmd\section{\Needspace*{4\baselineskip}}{}{}
\makeatletter

\usepackage[pdftex,bookmarks=false,colorlinks=true,citecolor=darkspringgreen,debug=true,
  naturalnames=true,pdfnewwindow=true]{hyperref}

\newtheorem{thm}{Theorem}[subsection]

\newtheorem{cor}[thm]{Corollary}

\newtheorem{lem}[thm]{Lemma}
\newtheorem{prop}[thm]{Proposition}

\theoremstyle{definition}
\newtheorem{defn}[thm]{Definition}

\theoremstyle{remark}
\newtheorem{rem}[thm]{Remark}

\newcommand{\nc}{\newcommand}
\nc{\renc}{\renewcommand} \nc{\ssec}{\subsection}
\nc{\sssec}{\subsubsection}

\nc{\on}{\operatorname} \nc{\wh}{\widehat}
\nc\ol{\overline} \nc\ul{\underline} \nc\wt{\widetilde}

\newcommand{\red}[1]{{\color{red}#1}}

\nc{\BA}{{\mathbb{A}}} \nc{\BC}{{\mathbb{C}}} \nc{\BF}{{\mathbb{F}}}
\nc{\BD}{{\mathbb{D}}} \nc{\BG}{{\mathbb{G}}} \nc{\BQ}{{\mathbb{Q}}}
\nc{\BM}{{\mathbb{M}}} \nc{\BN}{{\mathbb{N}}} \nc{\BO}{{\mathbb{O}}}
\nc{\BP}{{\mathbb{P}}} \nc{\BR}{{\mathbb{R}}}
\nc{\BZ}{{\mathbb{Z}}} \nc{\BS}{{\mathbb{S}}} \nc{\BW}{{\mathbb{W}}}

\nc{\CA}{{\mathcal{A}}} \nc{\CB}{{\mathcal{B}}} \nc{\CalC}{{\mathcal{C}}} \nc{\CalD}{{\mathcal{D}}}
\nc{\CE}{{\mathcal{E}}} \nc{\CF}{{\mathcal{F}}}
\nc{\CG}{{\mathcal{G}}} \nc{\CH}{{\mathcal{H}}}
\nc{\CI}{{\mathcal{I}}} \nc{\CJ}{{\mathcal{J}}} \nc{\CK}{{\mathcal{K}}} \nc{\CL}{{\mathcal{L}}}
\nc{\CM}{{\mathcal{M}}} \nc{\CN}{{\mathcal{N}}}
\nc{\CO}{{\mathcal{O}}} \nc{\CP}{{\mathcal{P}}}
\nc{\CQ}{{\mathcal{Q}}} \nc{\CR}{{\mathcal{R}}}
\nc{\CS}{{\mathcal{S}}} \nc{\CT}{{\mathcal{T}}}
\nc{\CU}{{\mathcal{U}}} \nc{\CV}{{\mathcal{V}}}  \nc{\CY}{{\mathcal Y}}
\nc{\CW}{{\mathcal{W}}} \nc{\CZ}{{\mathcal{Z}}}

\nc{\scrM}{{\mathscr M}} \nc{\scrP}{{\mathscr P}}

\nc{\cM}{{\check{\mathcal M}}{}} \nc{\csM}{{\check{\mathcal A}}{}}
\nc{\oM}{{\overset{\circ}{\mathcal M}}{}}
\nc{\obM}{{\overset{\circ}{\mathbf M}}{}}
\nc{\oCA}{{\overset{\circ}{\mathcal A}}{}}
\nc{\obA}{{\overset{\circ}{\mathbf A}}{}}
\nc{\ooM}{{\overset{\circ}{M}}{}}
\nc{\osM}{{\overset{\circ}{\mathsf M}}{}}
\nc{\vM}{{\overset{\bullet}{\mathcal M}}{}}
\nc{\nM}{{\underset{\bullet}{\mathcal M}}{}}
\nc{\oD}{{\overset{\circ}{\mathcal D}}{}}
\nc{\obD}{{\overset{\circ}{\mathbf D}}{}}
\nc{\oA}{{\overset{\circ}{\mathbb A}}{}}
\nc{\op}{{\overset{\bullet}{\mathbf p}}{}}
\nc{\cp}{{\overset{\circ}{\mathbf p}}{}}
\nc{\oU}{{\overset{\bullet}{\mathcal U}}{}}
\nc{\ofZ}{{\overset{\circ}{\mathfrak Z}}{}}

\nc{\ff}{{\mathfrak{f}}} \nc{\fv}{{\mathfrak{v}}}
\nc{\fa}{{\mathfrak{a}}} \nc{\fb}{{\mathfrak{b}}}
\nc{\fd}{{\mathfrak{d}}} \nc{\fe}{{\mathfrak{e}}}
\nc{\fg}{{\mathfrak{g}}} \nc{\fgl}{{\mathfrak{gl}}}
\nc{\fh}{{\mathfrak{h}}} \nc{\fri}{{\mathfrak{i}}}
\nc{\fj}{{\mathfrak{j}}} \nc{\fk}{{\mathfrak{k}}} \nc{\fl}{{\mathfrak{l}}}
\nc{\fm}{{\mathfrak{m}}} \nc{\fn}{{\mathfrak{n}}}
\nc{\ft}{{\mathfrak{t}}} \nc{\fu}{{\mathfrak{u}}}
\nc{\fw}{{\mathfrak{w}}} \nc{\fz}{{\mathfrak{z}}}
\nc{\fp}{{\mathfrak{p}}} \nc{\fq}{{\mathfrak{q}}} \nc{\frr}{{\mathfrak{r}}}
\nc{\fs}{{\mathfrak{s}}} \nc{\fsl}{{\mathfrak{sl}}}
\nc{\fso}{{\mathfrak{so}}} \nc{\fsp}{{\mathfrak{sp}}} \nc{\osp}{{\mathfrak{osp}}}
\nc{\hsl}{{\widehat{\mathfrak{sl}}}}
\nc{\hgl}{{\widehat{\mathfrak{gl}}}}
\nc{\hg}{{\widehat{\mathfrak{g}}}}
\nc{\chg}{{\widehat{\mathfrak{g}}}{}^\vee}
\nc{\hn}{{\widehat{\mathfrak{n}}}}
\nc{\chn}{{\widehat{\mathfrak{n}}}{}^\vee}

\nc{\fA}{{\mathfrak{A}}} \nc{\fB}{{\mathfrak{B}}} \nc{\fC}{{\mathfrak{C}}}
\nc{\fD}{{\mathfrak{D}}} \nc{\fE}{{\mathfrak{E}}}
\nc{\fF}{{\mathfrak{F}}} \nc{\fG}{{\mathfrak{G}}} \nc{\fH}{{\mathfrak{H}}}
\nc{\fI}{{\mathfrak{I}}} \nc{\fJ}{{\mathfrak{J}}}
\nc{\fK}{{\mathfrak{K}}} \nc{\fL}{{\mathfrak{L}}}
\nc{\fM}{{\mathfrak{M}}} \nc{\fN}{{\mathfrak{N}}}
\nc{\frP}{{\mathfrak{P}}} \nc{\fQ}{{\mathfrak{Q}}}
\nc{\fS}{{\mathfrak{S}}} \nc{\fT}{{\mathfrak{T}}} \nc{\fU}{{\mathfrak{U}}}
\nc{\fV}{{\mathfrak{V}}} \nc{\fW}{{\mathfrak{W}}}
\nc{\fX}{{\mathfrak{X}}} \nc{\fY}{{\mathfrak{Y}}}
\nc{\fZ}{{\mathfrak{Z}}}

\nc{\ba}{{\mathbf{a}}}
\nc{\bb}{{\mathbf{b}}} \nc{\bc}{{\mathbf{c}}} \nc{\be}{{\mathbf{e}}}
\nc{\bg}{{\mathbf{g}}} \nc{\bj}{{\mathbf{j}}} \nc{\bm}{{\mathbf{m}}}
\nc{\bn}{{\mathbf{n}}} \nc{\bp}{{\mathbf{p}}}
\nc{\bq}{{\mathbf{q}}} \nc{\br}{{\mathbf{r}}} \nc{\bs}{{\mathbf{s}}}
\nc{\bt}{{\mathbf{t}}} \nc{\bfu}{{\mathbf{u}}} \nc{\bv}{{\mathbf{v}}}
\nc{\bx}{{\mathbf{x}}} \nc{\by}{{\mathbf{y}}} \nc{\bz}{{\mathbf{z}}}
\nc{\bw}{{\mathbf{w}}} \nc{\bA}{{\mathbf{A}}}
\nc{\bB}{{\mathbf{B}}} \nc{\bC}{{\mathbf{C}}}
\nc{\bD}{{\mathbf{D}}} \nc{\bF}{{\mathbf{F}}} \nc{\bG}{{\mathbf{G}}}
\nc{\bH}{{\mathbf{H}}} \nc{\bI}{{\mathbf{I}}} \nc{\bJ}{{\mathbf{J}}}
\nc{\bK}{{\mathbf{K}}} \nc{\bM}{{\mathbf{M}}} \nc{\bN}{{\mathbf{N}}}
\nc{\bO}{{\mathbf{O}}} \nc{\bS}{{\mathbf{S}}} \nc{\bT}{{\mathbf{T}}}
\nc{\bU}{{\mathbf{U}}} \nc{\bV}{{\mathbf{V}}} \nc{\bW}{{\mathbf{W}}}
\nc{\bX}{{\mathbf{X}}}
\nc{\bY}{{\mathbf{Y}}} \nc{\bP}{{\mathbf{P}}}
\nc{\bZ}{{\mathbf{Z}}} \nc{\bh}{{\mathbf{h}}}

\nc{\sA}{{\mathsf{A}}} \nc{\sB}{{\mathsf{B}}}
\nc{\sC}{{\mathsf{C}}} \nc{\sD}{{\mathsf{D}}}
\nc{\sE}{{\mathsf{E}}} \nc{\sF}{{\mathsf{F}}} \nc{\sG}{{\mathsf{G}}}
\nc{\sI}{{\mathsf{I}}} \nc{\sK}{{\mathsf{K}}} \nc{\sL}{{\mathsf{L}}}
\nc{\sfm}{{\mathsf{m}}} \nc{\sM}{{\mathsf{M}}} \nc{\sN}{{\mathsf{N}}}
\nc{\sO}{{\mathsf{O}}} \nc{\sQ}{{\mathsf{Q}}} \nc{\sP}{{\mathsf{P}}}
\nc{\sT}{{\mathsf{T}}} \nc{\sZ}{{\mathsf{Z}}}
\nc{\sV}{{\mathsf{V}}} \nc{\sW}{{\mathsf{W}}}
\nc{\sfp}{{\mathsf{p}}} \nc{\sq}{{\mathsf{q}}} \nc{\sr}{{\mathsf{r}}}
\nc{\sfs}{{\mathsf{s}}} \nc{\st}{{\mathsf{t}}} \nc{\sfb}{{\mathsf{b}}}
\nc{\sfc}{{\mathsf{c}}} \nc{\sd}{{\mathsf{d}}}
\nc{\sz}{{\mathsf{z}}}

\nc{\tA}{{\widetilde{\mathbf{A}}}}
\nc{\tB}{{\widetilde{\mathcal{B}}}}
\nc{\tg}{{\widetilde{\mathfrak{g}}}} \nc{\tG}{{\widetilde{G}}}
\nc{\TM}{{\widetilde{\mathbb{M}}}{}}
\nc{\tO}{{\widetilde{\mathsf{O}}}{}}
\nc{\tU}{{\widetilde{\mathfrak{U}}}{}} \nc{\TZ}{{\tilde{Z}}}
\nc{\tx}{{\tilde{x}}} \nc{\tbv}{{\tilde{\bv}}}
\nc{\tfP}{{\widetilde{\mathfrak{P}}}{}} \nc{\tz}{{\tilde{\zeta}}}
\nc{\tmu}{{\tilde{\mu}}}

\nc{\urho}{\underline{\rho}} \nc{\uB}{\underline{B}}
\nc{\uC}{{\underline{\mathbb{C}}}} \nc{\ui}{\underline{i}}
\nc{\uj}{\underline{j}} \nc{\ofP}{{\overline{\mathfrak{P}}}}
\nc{\oB}{{\overline{\mathcal{B}}}}
\nc{\og}{{\overline{\mathfrak{g}}}} \nc{\oI}{{\overline{I}}}

\nc{\eps}{\varepsilon} \nc{\hrho}{{\hat{\rho}}}
\nc{\blambda}{{\boldsymbol{\lambda}}} \nc{\bmu}{{\boldsymbol{\mu}}} \nc{\bnu}{{\boldsymbol{\nu}}}
\nc{\btheta}{{\boldsymbol{\theta}}} \nc{\bzeta}{{\boldsymbol{\zeta}}} \nc{\bta}{{\boldsymbol{\eta}}}
\nc{\bomega}{{\boldsymbol{\omega}}} \nc{\bxi}{{\boldsymbol{\xi}}} \nc{\brho}{{\boldsymbol{\rho}}}

\nc{\one}{{\mathbf{1}}} \nc{\two}{{\mathbf{t}}}

\nc{\Sym}{\mathop{\operatorname{\rm Sym}}}
\nc{\Tot}{{\mathop{\operatorname{\rm Tot}}}}
\nc{\Spec}{\mathop{\operatorname{\rm Spec}}}
\nc{\Ker}{{\mathop{\operatorname{\rm Ker}}}}
\nc{\Isom}{{\mathop{\operatorname{\rm Isom}}}}
\nc{\Hilb}{{\mathop{\operatorname{\rm Hilb}}}}
\nc{\deeq}{{\mathop{\operatorname{\rm deeq}}}}
\nc{\End}{{\mathop{\operatorname{\rm End}}}}
\nc{\Ext}{{\mathop{\operatorname{\rm Ext}}}}
\nc{\Hom}{{\mathop{\operatorname{\rm Hom}}}}
\nc{\CHom}{{\mathop{\operatorname{{\mathcal{H}}\it om}}}}
\nc{\GL}{{\mathop{\operatorname{\rm GL}}}}
\nc{\SL}{{\mathop{\operatorname{\rm SL}}}}
\nc{\SO}{{\mathop{\operatorname{\rm SO}}}}
\nc{\Sp}{{\mathop{\operatorname{\rm Sp}}}}
\nc{\OSp}{{\mathop{\operatorname{\rm SOSp}}}}
\nc{\gr}{{\mathop{\operatorname{\rm gr}}}}
\nc{\Id}{{\mathop{\operatorname{\rm Id}}}}
\nc{\perf}{{\mathop{\operatorname{\rm perf}}}}
\nc{\defi}{{\mathop{\operatorname{\rm def}}}}
\nc{\length}{{\mathop{\operatorname{\rm length}}}}
\nc{\supp}{{\mathop{\operatorname{\rm supp}}}}
\nc{\HC}{{\mathcal H}{\mathcal C}}

\nc{\pr}{{\operatorname{pr}}}
\nc{\Cliff}{{\mathsf{Cliff}}}
\nc{\loc}{{\operatorname{loc}}}
\nc{\Fl}{{\mathbf{Fl}}} \nc{\Ffl}{{\mathcal{F}\ell}}
\nc{\Fib}{{\mathsf{Fib}}}
\nc{\Coh}{{\mathsf{Coh}}} \nc{\FCoh}{{\mathsf{FCoh}}}
\nc{\Perf}{{\mathsf{Perf}}}
\nc{\wtimes}{\mathbin{\widetilde\times}}

\nc{\reg}{{\text{\rm reg}}}
\nc{\self}{{\text{\rm self}}}
\nc{\gvee}{{\mathfrak g}^{\!\scriptscriptstyle\vee}}
\nc{\tvee}{{\mathfrak t}^{\!\scriptscriptstyle\vee}}
\nc{\nvee}{{\mathfrak n}^{\!\scriptscriptstyle\vee}}
\nc{\bvee}{{\mathfrak b}^{\!\scriptscriptstyle\vee}}
       \nc{\rhovee}{\rho^{\!\scriptscriptstyle\vee}}

\nc{\cplus}{{\mathbf{C}_+}} \nc{\cminus}{{\mathbf{C}_-}}
\nc{\cthree}{{\mathbf{C}_*}} \nc{\Qbar}{{\bar{Q}}}

\newcommand{\oC}{\vphantom{j^{X^2}}\smash{\overset{\circ}{\vphantom{\rule{0pt}{0.55em}}\smash{C}}}}
\newcommand{\oZ}{\vphantom{j^{X^2}}\smash{\overset{\circ}{\vphantom{\rule{0pt}{0.55em}}\smash{Z}}}}
\newcommand{\oCI}{\vphantom{j^{X^2}}\smash{\overset{\circ}{\vphantom{\rule{0pt}{0.55em}}\smash{\mathcal I}}}}
\newcommand{\bCD}{\vphantom{j^{X^2}}\smash{\overset{\bullet}{\vphantom{\rule{0pt}{0.55em}}\smash{\mathcal D}}}}
\newcommand{\bCP}{\vphantom{j^{X^2}}\smash{\overset{\bullet}{\vphantom{\rule{0pt}{0.55em}}\smash{\mathcal P}}}}

\newcommand\iso{\mathbin{\vphantom{j^{X^2}}\smash{\overset{\sim}{\vphantom{\rule{0pt}{0.20em}}\smash{\longrightarrow}}}}}
\nc{\Gtimes}{\vphantom{j^{X^2}}\smash{\overset{G}{\vphantom{\rule{0pt}{0.30em}}\smash{\times}}}}
\nc{\sGtimes}{\vphantom{j^{X^2}}\smash{\overset{\mathsf G}{\vphantom{\rule{0pt}{0.30em}}\smash{\times}}}}

\nc{\bOmega}{{\overline{\Omega}}}

\nc{\seq}[1]{\stackrel{#1}{\sim}}

\nc{\aff}{{\operatorname{aff}}}
\nc{\fin}{{\operatorname{fin}}}
\nc{\mir}{{\operatorname{mir}}}
\nc{\triv}{{\operatorname{triv}}}
\nc{\ext}{{\operatorname{ext}}}
\nc{\righ}{{\operatorname{right}}}
\nc{\lef}{{\operatorname{left}}}
\nc{\forg}{{\operatorname{forg}}}
\nc{\fid}{{\operatorname{fd}}}
\nc{\odd}{{\operatorname{odd}}}
\nc{\even}{{\operatorname{even}}}
\nc{\modu}{{\operatorname{-mod}}}
\nc{\Gr}{{\mathbf{Gr}}}
\nc{\FT}{{\operatorname{FT}}}
\nc{\Mat}{{\operatorname{Mat}}}
\nc{\MSt}{{\operatorname{MSt}}}
\nc{\sph}{{\operatorname{sph}}}
\nc{\GR}{{\mathbf{GR}}}
\nc{\Perv}{{\operatorname{Perv}}}
\nc{\Rep}{{\operatorname{Rep}}}
\nc{\Ind}{{\operatorname{Ind}}}
\nc{\IC}{{\operatorname{IC}}}
\nc{\Bun}{{\operatorname{Bun}}}
\nc{\Proj}{{\operatorname{Proj}}}
\nc{\Stab}{{\operatorname{Stab}}}
\nc{\pt}{{\operatorname{pt}}}
\nc{\bfmu}{{\boldsymbol{\mu}}}
\nc{\bfomega}{{\boldsymbol{\omega}}}
\nc{\calM}{\mathcal M}
\nc{\calA}{\mathcal A}
\nc{\calO}{\mathcal O}
\nc{\CC}{\mathbb C}
\nc{\calN}{\mathcal N}
\nc{\grg}{\mathfrak g}
\nc{\dslash}{/\!\!/}
\nc{\tslash}{/\!\!/\!\!/}
\nc\grt{\mathfrak t}
\nc\bfM{\mathbf M}
\nc\bfN{\mathbf N}
\nc\Sig{\Sigma}
\nc\ZZ{\mathbb{Z}}
\nc\calC{\mathcal C}
\nc\calF{\mathcal F}
\nc\calX{\mathcal X}
\nc\calY{\mathcal Y}
\nc\QCoh{\operatorname{QCoh}}
\nc\IndCoh{\operatorname{IndCoh}}
\nc\Maps{\operatorname{Maps}}
\nc\Dmod{D-\operatorname{mod}}
\newcommand\Hecke{\operatorname{Hecke}}
\nc{\calD}{\mathcal D}
\nc\bfO{\mathbf O}

\nc\GG{\mathbb G}
\nc\calK{\mathcal K}
\nc{\calG}{\mathcal G}
\nc\RHom{\operatorname{RHom}}
\nc\Res{\operatorname{Res}}
\nc\Av{\operatorname{Av}}
\nc\grs{\mathfrak s}
\nc{\tilX}{\widetilde X}
\nc\calB{\mathcal B}
\nc\calS{\mathcal S}
\nc\calT{\mathcal T}
\nc\calZ{\mathcal Z}
\nc\LS{\operatorname{LocSys}}
\nc\bfL{\on{\mathbf L}}

\newcommand*\circled[1]
{*{#1}}
\makeatletter
\newcommand{\raisemath}[1]{\mathpalette{\raisem@th{#1}}}
\newcommand{\raisem@th}[3]{\raisebox{#1}{$#2#3$}}
\nc{\binlim}[2][]{\def\@tempa{#1}\@ifnextchar^{\@binlim{#2}}{\@binlim{#2}^{}}}
\def\@binlim#1^#2{\mathbin{\@ifempty{#2}{\mathop{#1}}{\mathop{#1}\@xp\displaylimits\@tempa^{#2}}}}

\makeatother

\nc\cX{{\mathcal X}}

\newcommand{\dbkts}[1]{[\![#1]\!]}
\newcommand{\dprts}[1]{(\!(#1)\!)}
\nc\Gm{{\mathbb G_m}}
\nc{\isoto}
    {\stackrel{\displaystyle\raisemath{-.2ex}{\sim}}\to}
\nc\cD\CalD
\nc\setm\setminus
\nc\cE\CE
\renc\Hecke{\mathit{\CH\kern-.2ex ecke}}
\nc\bsl\backslash
\nc\Fq{\mathbb F_q}
\nc\x\times
\nc\bGO{{\bG_\bO}}
\nc\bla\blambda
\nc\blt\bullet
\nc\opp{{\on{op}}}
\nc\srel\stackrel
\nc\oast\circledast
\nc\tbx{\binlim{\widetilde\boxtimes{}}}
\nc\into\hookrightarrow
\nc\otto\leftrightarrow
\renc\setminus\smallsetminus
\nc\phitau{\varphi\tau}
\newenvironment{i-ii-iii}{%
\begin{enumerate}
}%
{\end{enumerate}}
\nc\ceil[1]{\lceil#1\rceil}  \nc\floor[1]{\lfloor#1\rfloor}
\nc\Lie{\on{Lie}}
\nc\sS{{\mathsf S}}
\nc\vvv{\ensuremath{\red\surd}}
\nc\la\lambda
 \let\arXiv\arxiv
\nc\kap{\kappa}
\nc\gra{\mathfrak a}
\nc\gl{\mathfrak{gl}}
\nc\sTr{\operatorname{sTr}}
\nc\hatG{\widehat{G}}
\nc\calL{\mathcal L}
\nc\Whit{\operatorname{Whit}}
\nc\KL{\operatorname{KL}}

\makeatletter
\renewcommand{\subsection}{\@startsection{subsection}{2}{0pt}{-3ex
plus -1ex minus -0.2ex}{-2mm plus -0pt minus
-2pt}{\normalfont\bfseries}} \makeatother

\numberwithin{equation}{subsection}
\allowdisplaybreaks
\nc\mto{\mapsto }
\nc\en{\enspace }

\begin{document}

\author[A.Braverman]{Alexander Braverman}
\address{Department of Mathematics, University of Toronto and Perimeter Institute
of Theoretical Physics, Waterloo, Ontario, Canada, N2L 2Y5}
\email{braval@math.toronto.edu}

\author[M.Finkelberg]{Michael Finkelberg}
\address{Einstein Institute of Mathematics, The Hebrew University of Jerusalem,
  Edmond J. Safra Campus, Giv’at Ram, Jerusalem, 91904, Israel;
\newline  National Research University Higher School of Economics;
\newline Skolkovo Institute of Science and Technology}
\email{fnklberg@gmail.com}

\author[R.Travkin]{Roman Travkin}
\address{Skolkovo Institute of Science and Technology, Moscow, Russia}
\email{roman.travkin2012@gmail.com}

\title
    {Gaiotto conjecture for $\Rep_q(\GL(N-1|N))$}
\dedicatory{To David Kazhdan and George Lusztig on their 75th birthdays with admiration}
    \maketitle

    \begin{abstract}
      We prove D.~Gaiotto's conjecture about geometric Satake equivalence for quantum supergroup
      $U_q(\fgl(N-1|N))$ for generic $q$. The equivalence goes through the category of factorizable
      sheaves.
    \end{abstract}

    \section{Introduction}

    \subsection{Geometric Satake equivalence and FLE}
    Let $\bF=\BC\dprts{t}\supset\BC\dbkts{t}=\nobreak\bO$. Let $G$ be a connected reductive group over $\BC$.
    Let $\Gr_G=G(\bF)/G(\bO)$
    be the affine Grassmannian of $G$. One can consider the category $\Perv_{G(\bO)}(\Gr_G)$ of
    $G(\bO)$-equivariant perverse sheaves on
    $\Gr_G$. This is a tensor category over $\BC$. The geometric Satake equivalence identifies this category with the category
    $\Rep(G^{\vee})$ of finite-dimensional representations of the Langlands dual group $G^{\vee}$.

    The above equivalence is very important for many applications (e.g.\ it is in some sense the
    starting point for the geometric Langlands correspondence) but at the same time it has two
    serious drawbacks:

    1) It does not hold on the derived level. In fact, the derived Satake equivalence~\cite{bef} does provide a description of the derived category $D_{G(\bO)}(\Gr_G)$ in terms of $G^{\vee}$ but the answer is certainly not the derived category of $\Rep(G^{\vee})$.

    2) For many reasons it would be nice to generalize the above equivalence so that the category $\Rep(G^{\vee})$ gets replaced with the category $\Rep_q(G^{\vee})$ --- the category of finite-dimensional representations of the corresponding quantum group. But it seems that it is impossible to find such a generalization.

On the other hand, J.~Lurie and D.~Gaitsgory found a replacement of the geometric Satake equivalence
(called the Fundamental Local Equivalence, or FLE) where both of the above problems disappear.

Namely, let $U$ be a maximal unipotent subgroup of $G$, and let $\overline{\chi}\colon U\to \GG_a$
be its generic character. Let $\chi\colon U(\bF)\to \GG_a$ be given by the formula
$\chi(u(t))=\Res_{t=0} \overline{\chi}(u(t))$.
Let now $\Whit(\Gr_G)$ be the derived category of $(U(\bF),\chi)$-equivariant sheaves on $\Gr_G$.
Then this category is equivalent to $D(\Rep(G^{\vee}))$. Moreover, this statement can be generalized
to an equivalence between the category $D(\Rep_q(G^{\vee}))$ and the corresponding
category $\Whit_q(\Gr_G)$ (sheaves twisted by the corresponding complex power of a certain
determinant line bundle on $\Gr_G$ --- we refer the reader to~\cite{gl} for the discussion of
precise meaning of $q$ etc.). This equivalence preserves the natural $t$-structures on both sides.

\subsection{Gaiotto conjectures}
Let now $G=\GL(N)$. In this case D.~Gaiotto constructed certain series of subgroups of $G$ endowed
with an additive character, which in many respects resemble the pair $(U,\overline{\chi})$ above.
Namely, fix $M<N$. Consider the natural embedding of $\GL(M)$ into $\GL(N)$. Then one can construct
(see~\cite[\S2]{bfgt}) a unipotent subgroup $U_{M,N}$ of $\GL(N)$ that is normalized by $\GL(M)$,
and a character $\overline{\chi}_{M,N}\colon U_{M,N}\to \GG_a$ which fixed by the adjoint action of
$\GL(M)$ such that conjecturally for generic $q$ we have an equivalence (the notation is explained below):
\begin{equation}\label{gaiotto}
SD_{\GL(M,\bO)\ltimes (U_{M,N}(\bF),\chi_{M,N}),q}(\bCD)\simeq D(\Rep_q(\GL(M|N)),
\end{equation}
that respects the natural $t$-structures on both sides, i.e.\ it should induce an equivalence
\begin{equation}\label{gaiotto-perv}
S\Perv_{\GL(M,\bO)\ltimes (U_{M,N}(\bF),\chi_{M,N}),q}(\bCD)\simeq \Rep_q(\GL(M|N)).
\end{equation}

Here the notations are as follows:
a) $\mathcal D$ stands for certain determinant line bundle on $\Gr_{\GL(N)}$ and $\bCD$ is the total space of this line bundle with zero section
removed.

b) $SD_{\GL(M,\bO)\ltimes (U_{M,N}(\bF),\chi_{M,N}),q}(\bCD)$ stands for the derived category of
$\GL(M,\bO)\ltimes (U_{M,N}(\bF),\chi_{M,N})$-equivariant $q$-monodromic sheaves $\bCD$ with coefficients in super-vector spaces; $S\Perv_{\GL(M,\bO)\ltimes (U_{M,N}(\bF),\chi_{M,N}),q}(\bCD)$ stands for the corresponding category of perverse sheaves.

c) $\GL(M|N)$ is the super group of automorphisms of the super vector space $\BC^{M|N}$ and
$\Rep_q(GL(M|N))$ is the category of finite-dimensional representations of the corresponding
quantum group (cf.~\S\ref{quantum} for the precise definitions).

Let us note that the above formulation is for generic $q$; a similar formulation should hold for
all $q$, but one has to be more careful about the precise form of the corresponding quantum
super group over $\BC[q,q^{-1}]$.

\subsection{What is done in this paper?}
In this paper we deal with the case $M=N-1$ for generic $q$. The advantage of the $M=N-1$ assumption is that in this case the group $U_{M,N}$ is
trivial, so $\GL(M,\bO)\ltimes U_{M,N}(\bF)$ is just equal to $\GL(N-1,\bO)$ (and the character $\chi$ is trivial as well).
The current  paper should be thought of as a sequel to~\cite{bfgt}. There we consider (among other things) the case $q=1$. As was noted above, one has to be careful about specializing to non-generic $q$. It turns out that for $q=1$ the correct statement is as follows.

Consider a degenerate version $\ul\fgl(N-1|N)$ where the
supercommutator of the even elements (with even or odd elements) is the same as in $\fgl(N-1|N)$,
while the supercommutator of any two odd elements is set to be zero.
In other words, the even part $\ul\fgl(N-1|N){}_{\bar0}=\fgl_{N-1}\oplus\fgl_N$ acts naturally on the
odd part $\ul\fgl(N-1|N){}_{\bar1}=\Hom(\BC^{N-1},\BC^N)\oplus\Hom(\BC^N,\BC^{N-1})$, but the
supercommutator $\ul\fgl(N-1|N){}_{\bar1}\times\ul\fgl(N-1|N){}_{\bar1}\to\ul\fgl(N-1|N){}_{\bar0}$
equals zero.

The category of finite dimensional representations of the corresponding supergroup
$\ul\GL(N-1|N)$ is denoted $\Rep(\ul\GL(N-1|N))$. In~\cite{bfgt}
we construct a tensor equivalence from the abelian category $S\Perv_{\GL(N-1,\bO)}(\Gr_{\GL_N})$
of equivariant perverse sheaves with coefficients in super vector spaces
to $\Rep(\ul\GL(N-1|N))$. Here
the monoidal structure on $S\Perv_{\GL(N-1,\bO)}(\Gr_{\GL_N})$ is defined via the fusion product
(nearby cycles in the Beilinson-Drinfeld Grassmannian).
This equivalence is reminiscent of the classical
geometric Satake equivalence $\Perv_{\GL(N,\bO)}(\Gr_{\GL_N})\cong\Rep(\GL_N)$, but as was noted above it should rather be thought of as analog of FLE. In particular, in~\cite{bfgt} we also prove the corresponding derived equivalence
\[SD_{\GL(N-1,\bO)}(\Gr_{\GL_N})\simeq D(\Rep(\ul\GL(N-1|N)))\]
(in fact, we first prove the derived equivalence and then show that it is compatible with the $t$-structures on both sides).

The main purpose of this paper is to prove~\eqref{gaiotto-perv} for $M=N-1$~(Theorem~\ref{main}).
In other words, assuming that $q$ is transcendental\footnote{Probably the assumption that $q$ is not
a root of unity should suffice but for our current proof we need to assume that $q$ is transcendental
for certain technical reasons.}
we prove a braided monoidal equivalence between abelian categories $S\Perv_{\GL(N-1,\bO),q}(\bCD)$ and
$\Rep_q(\GL(N-1|N))$.
The braided tensor structure on the geometric side is again defined via the fusion product.
Contrary to the case $q=1$, we use the abelian equivalence~\eqref{gaiotto-perv} to derive the
derived equivalence~\eqref{gaiotto}. It follows from an equivalence
$D(S\Perv_{\GL(N-1,\bO),q}(\bCD))\simeq SD_{\GL(N-1,\bO),q}(\bCD)$~(Theorem~\ref{derived}).

\begin{rem} Let us note that the $q=1$ case discussed above is a special case of a very general set of conjectures due to D.~Ben-Zvi, Y.~Sakellaridis and A.~Venkatesh; those conjectures were in fact motivated by known results about automorphic $L$-functions. However, to the best of our knowledge, it is not known how to extend those general conjectures to the ``quantum'' (i.e.\ general $q$) case. Thus in some sense at the moment the only motivation for the equivalences (\ref{gaiotto}) and (\ref{gaiotto-perv}) comes from mathematical physics.
  \end{rem}

    \subsection{Outline of the proof of the main theorem}
Our argument follows the scheme of D.~Gaitsgory's proof~\cite{g} of the FLE for generic $q$.
We use the Lurie-Gaitsgory generalization of~\cite{bfs}: a braided tensor equivalence between
$\Rep_q(\GL(N-1|N))$ and an appropriate category $\on{FS}$ of factorizable sheaves on
configuration spaces of a smooth projective curve $C$. In the main body of the paper we
construct a braided tensor equivalence $F\colon S\Perv_{\GL(N-1,\bO),q}(\bCD)\iso\on{FS}$.

To this end we use a correspondence between the Hecke stack $\GL(N-1,\bO)\backslash\Gr_{\GL_N}$
and a configuration space of $C$. It is nothing but the zastava model $\CW$ with poles
introduced in~\cite{sw}. In fact, Y.~Sakellaridis and J.~Wang worked out their zastava models
for arbitrary affine spherical varieties with all the spherical roots of type $T$, and we
use their theory for one particular spherical variety $H\backslash G$ where
$G=\GL_{N-1}\times\GL_N$, and $H$ is the block-diagonally embedded $\GL_{N-1}$.
Note that the Sakellaridis-Wang (SW for short) zastava models for $H\backslash G$ differ
drastically from their classical counterparts for $\overline{U\backslash G}$ (the base affine
space). The main difference is that while the factorization morphism to configuration space
of $C$ looks like an integrable system in the classical case, for $H\backslash G$ this
factorization morphism is semismall.

Due to the above semismallness, the functor $F$ (defined as the push-pull via the zastava
with poles correspondence) is automatically exact. More precisely, we also use the smoothness
of the morphism from zastava $\CW$ to the Hecke stack $\GL(N-1,\bO)\backslash\Gr_{\GL_N}$ and
the cleanness property of extension from $\CW$ to compactified zastava $\ol\CW$.
A more refined analysis of our twisting on the fibers of the factorization morphism allows
us to strengthen the above exactness result and prove that $F$ takes irreducible sheaves
in the Gaiotto category to irreducible factorizable sheaves. Furthermore, since the braided
tensor structures on both categories $S\Perv_{\GL(N-1,\bO),q}(\bCD)$ and $\on{FS}$ are defined
via the same fusion construction, our functor $F$ is automatically braided tensor.

In the FLE for generic $q$ case of~\cite{g} this was the end of the story since the categories
in question were semisimple. We need a little extra work. Namely, we need to exhibit enough
projectives in the Gaiotto category that go to projective factorizable sheaves. The corresponding
projectives in $\Rep_q(\GL(N-1|N))$ are tensor products $V_{\bmu,\bnu}\otimes V_{\bzeta,\brho}$
of irreducibles for a particular {\em typical} highest weight (a bisignature) $(\bzeta,\brho)$.
The typicality assumption guarantees that $V_{\bzeta,\brho}$ is projective, and hence
$V_{\bmu,\bnu}\otimes V_{\bzeta,\brho}$ is projective as well by the rigidity property of the tensor
category $\Rep_q(\GL(N-1|N))$. Our remaining task is to mimick this construction in the
Gaiotto category.

First, the rigidity property of the monodromic sheaf
$\IC^q_{\bzeta,\brho}\in S\Perv_{\GL(N-1,\bO),q}(\bCD)$ can be deduced from the known rigidity for
$q=1$~\cite{bfgt} by a deformation argument. It is here that the assumption of Weil genericity
of $q$ (i.e.\ $q$ is transcendental) would be used. We actually take another route explained
to us by P.~Etingof. We consider a full abelian tensor subcategory
$\CE\subset S\Perv_{\GL(N-1,\bO),q}(\bCD)$
generated by the sheaves $\IC^q_{\on{taut}},(\IC^q_{\on{taut}})^*$ corresponding to the tautological
representation of $U_q(\fgl(N-1|N))$ and its dual. Following~\cite{c}, one can prove that
$\CE$ is braided tensor equivalent to $\Rep_q(\GL(N-1|N))$. But here as well we need to
use the Weil genericity of $q$ assumption. Anyway, $\CE$ is rigid and contains all the
irreducible sheaves $\IC^q_{\bmu,\bnu}$.

Finally, we need to establish the projectivity of $\IC^q_{\bzeta,\brho}$. This is proved by
checking that any other irreducible sheaf $\IC^q_{\bmu,\bnu}$ has zero stalks at the
$\GL(N-1,\bO)$-orbit $\BO_{\bzeta,\brho}\subset\Gr_{\GL_N}$. To this end we use a special curve
$C=\BA^1$, and show that the appropriate zastava with poles spaces $\CW$ play the role of
transversal slices to $\BO_{\bzeta,\brho}$, similarly to~\cite[\S2(v)]{bfn}. The advantage of
our choice $C=\BA^1$ is that by the contraction principle, the stalk in question equals
the cohomology of $\CW$ with coefficients in the pull back of $\IC^q_{\bmu,\bnu}$.
By the definition of our functor $F$, the latter cohomology equals the cohomology of a
configuration space of $C=\BA^1$ with coefficients in the corresponding irreducible factorizable
sheaf. The latter cohomology can be computed as certain Ext in the category $\CO$ of
$U_q(\fgl(N-1|N))$. It vanishes since $(\bmu,\bnu)$ and $(\bzeta,\brho)$ lie in different
linkage classes of this category. This bootstrap argument finishes the proof of our main theorem.

\subsection{}
A few concluding remarks are in order. First, one can also prove the Gaiotto conjecture for $M=N$
and $\GL(N,\bO)$-equivariant
$q$-monodromic perverse sheaves on the determinant line bundle on the mirabolic affine Grassmannian
of $\GL_N$~\cite[\S2.5]{bfgt} as well as for orthosymplectic quantum groups and
$S\Perv_{\SO(N-1,\bO),q}(\bCD)$ ($q$-monodromic perverse sheaves on the punctured determinant line
bundle over $\Gr_{\SO_N}$)~\cite[\S3.2]{bft} along the same lines.

Second, the original statement of Gaiotto conjecture~\cite[\S2]{bfgt} was not in terms of
quantum supergroups, but in terms of representations of the corresponding affine Lie superalgebras.
The version discussed in the present paper is obtained via the (not yet established) super analogue
of the Kazhdan-Lusztig equivalence.

Third, similarly to the Iwahori version of FLE established in~\cite{ya}, we expect a derived
equivalence between the category of $q$-monodromic sheaves on (the determinant line bundle on)
the affine flag variety of $\GL(N)$
equivariant with respect to the Iwahori subgroup of $\GL(N-1,\bO)$, and the category $\CO$ of
the quantized universal enveloping algebra $U_q(\fgl(N-1|N))$. It would yield the Kazhdan-Lusztig
type formulas for the characters of
irreducibles in the category $\CO$ (in particular, the finite dimensional irreducibles).

\subsection{Acknowledgments}
It should be clear from the above that the present note (as well as the whole modern representation
theory) rests upon the fundamental discoveries made by D.~Kazhdan and G.~Lusztig. Our intellectual
debt to them cannot be overestimated.

This note is the result of generous explanations by I.~Entova-Aizenbud, P.~Etingof, D.~Gaiotto,
D.~Gaitsgory, D.~Leites, Y.~Sakellaridis, V.~Serganova and J.~Wang. We are also grateful to
I.~Shchepochkina and A.~Tsymbaliuk for the help with references and to R.~Yang
for the interesting discussions. Finally, we would like to thank the anonymous referee for his
numerous useful suggestions and corrections.

A.B.~was partially supported by NSERC. The research of M.F.~was supported by
the Israel Science Foundation (grant No.~994/24).

    \section{Quantum supergroups and factorizable sheaves}\label{quantum}

    In this section we briefly recall a geometric realization of the category of representations
    of the quantum supergroup $U_q(\fgl(N-1|N))$ in factorizable sheaves following~\cite{bfs,g}.

    \subsection{Quantum supergroup}
    We fix a transcendental complex number $q\in\BC$.
    For a definition of $U_q(\fgl(N-1|N))$ see~\cite{y,chw}. Note that the definition
    depends on a choice of a Borel subalgebra of $\fgl(N-1|N)$, but the resulting quantum
    algebras are all isomorphic according to~\cite[Proposition 7.4.1]{y2}. We will use the so
    called {\em mixed} Borel subalgebra all of whose simple roots are odd isotropic.
More precisely, we fix a basis
    $\delta_1,\ldots,\delta_{N-1},\varepsilon_1,\ldots,\varepsilon_N$ of diagonal entries weights
of the diagonal Cartan subgroup of $\GL(N-1|N)$.
The positive simple roots with respect to the mixed Borel subalgebra of $\GL(N-1|N)$ are
as follows:
\begin{equation}
  \label{mixed}
  \alpha_1=\varepsilon_1-\delta_1,\ \alpha_2=\delta_1-\varepsilon_2,\ldots,\
  \alpha_{2i-1}=\varepsilon_i-\delta_i,\ \alpha_{2i}=\delta_i-\varepsilon_{i+1},\ldots,\
  \alpha_{2N-2}=\delta_{N-1}-\varepsilon_N.
  \end{equation}

    In this case the defining relations (quantum analogues of Serre relations) of $U_q(\fgl(N-1|N))$
    are explicitly written down in~\cite[Lemma 6.1.1(i)]{y} and~\cite[Proposition 2.7(AB)]{chw}.

    According to~\cite[Theorem 48]{ge}, the highest weights of the irreducible finite dimensional
    representations of $U_q(\fgl(N-1|N))$
    (with respect to the mixed Borel subalgebra) are the same as the highest weights of the
    irreducible finite dimensional representations of non-quantized supergroup $\GL(N-1|N)$.
    Let us recall the classification of such highest weights. In the above basis
    $\delta_1,\ldots,\delta_{N-1},\varepsilon_1,\ldots,\varepsilon_N$, the weights are pairs
    $(\bmu,\bnu)\in\BZ^{N-1}\oplus\BZ^N=\BZ^{2N-1}=X$. The dominant highest weights for
    $\GL_{N-1}\x\GL_N$ are
    the pairs of signatures $\big(\blambda=(\lambda_1\geq\ldots\geq\lambda_{N-1}),\
    \btheta=(\theta_1\geq\ldots\geq\theta_N)\big)$ such that the length of $\blambda$
    (resp.\ $\btheta$) is $N-1$ (resp.\ $N$).

    \begin{lem}[V.~Serganova] A pair of signatures $(\blambda,\btheta)$ is the highest weight of
      an irreducible
      finite dimensional representation of $\GL(N-1|N)$ (and of the quantum supergroup
      $U_q(\fgl(N-1|N))$) if and only if the following condition holds:
\begin{equation}
      \label{serg}
      \on{if}\ \theta_i=\theta_{i+1},\ \on{then}\ \theta_i+\lambda_i=0;\ \&\
      \on{if}\ \lambda_{i-1}=\lambda_i,\ \on{then}\ \theta_i+\lambda_i=0.
\end{equation}
    \end{lem}

    \begin{proof}
The criterion in question follows from~\cite[Theorem 10.5]{s1}. Alternatively, it can be
deduced from~\cite[Corollary 8.6.2]{m}. Namely, we consider another Borel subalgebra
$\fb'$ with the same positive even roots but with a unique simple odd root
$\varepsilon_N-\delta_1$. Then {\em any} pair of signatures $(\bmu,\bnu)$ is a highest weight
(with respect to $\fb'$) of a finite-dimensional $\GL(N-1|N)$-module since the Kac module
$\Ind_{\GL(N-1)\times\GL(N)}^{\GL(N-1|N)}(V^\bmu\otimes V^\bnu)$ is finite-dimensional. It remains to
rewrite the highest weight $(\blambda,\btheta)$ with respect to the mixed Borel subalgebra $\fb$
in terms of the highest weight $(\bmu,\bnu)$ with respect to $\fb'$. To this end we use the
following sequence of odd reflections (we learned from A.~Lebedev) taking $\fb'$ to $\fb$,
where we number the vertices of the Dynkin graph from right to left:
\[r_{2N-2}(r_{2N-4}r_{2N-3})\cdots(r_{2k}r_{2k+1}\cdots r_{N+k-2}r_{N+k-1})\cdots
(r_4r_5\cdots r_Nr_{N+1})(r_2r_3\cdots r_{N-1}r_N).\]

    \end{proof}

\begin{defn}
  We denote by $\Rep_q(\GL(N-1|N))$ the abelian braided tensor category of finite dimensional
  representations of $U_q(\fgl(N-1|N))$ equipped with a grading by the weight lattice $X$ that
  defines the action of the Cartan subalgebra of $U_q(\fgl(N-1|N))$.
\end{defn}

\subsection{Configuration spaces}
We fix a smooth projective curve $C$ with a marked point $c\in C$. Given a weight
$(\bmu,\bnu)\in X=\BZ^{N-1}\oplus\BZ^N$, we consider the configuration space $C^{(\bmu,\bnu)}$
of $X$-colored divisors $D=-\sum_{i=1}^{N-1}\delta_i\Delta_i+\sum_{i=1}^N\varepsilon_iE_i$ on $C$
of total degree $(\bmu,\bnu)$ with the following positivity condition.

For $1\leq i\leq N-1$, we set
\[D_{2i}=\Delta_1+\ldots+\Delta_i-E_1-\ldots-E_i\ \on{and}\
D_{2i-1}=\Delta_1+\ldots+\Delta_{i-1}-E_1-\ldots-E_i-B,\]
where\footnote{$B$ stands for Berezinian.}
\[B=\Delta_1+\ldots+\Delta_{N-1}-E_1-\ldots-E_N.\]
These divisors are the coefficients of $D$ in the basis of negative simple roots:
\[D=-\sum_{j=1}^{2N-2}\alpha_iD_i+\left(\sum_{i=1}^{N-1}\delta_i-\sum_{i=1}^N\varepsilon_i\right)B.\]
Then we require
\begin{equation}
  \label{effective}
  B\ \on{is}\ \on{supported}\ \on{at}\ c\in C,\ \on{and}\ D_j\ \on{is}\ \on{effective}\ \on{away}\
  \on{from}\ c\in C\ \on{for} 1\leq j\leq 2N-2.
  \end{equation}


The space $C^{(\bmu,\bnu)}$ has a natural structure of an ind-variety. Namely, restricting the
degrees at $c\in C$ we get
\[C^{(\bmu,\bnu)}=\bigcup_{(\blambda,\btheta)}C^{(\bmu,\bnu)}_{\leq(\blambda,\btheta)},\]
where $D\in C^{(\bmu,\bnu)}_{\leq(\blambda,\btheta)}$ if $D-(\blambda,\btheta)\cdot c$ enjoys
the effectivity property~\eqref{effective} at {\em all} points of $C$.
We have $C^{(\bmu,\bnu)}_{\leq(\blambda,\btheta)}\simeq C^\alpha=\prod_{j=1}^{2N-2}C^{(a_j)}$, where
$(\blambda,\btheta)-(\bmu,\bnu)=\alpha=\sum_{j=1}^{2N-2}a_j\alpha_j$ for $a_j\in\BN$.

\subsection{A factorizable line bundle}
\label{factor bundle}
We consider a line bundle $\CP$ on $C^{(\bmu,\bnu)}$ with fibers
\[\CP_D=\bigotimes_{i=1}^{N-1}\det R\Gamma(C,\CO_C(-\Delta_i))\otimes\bigotimes_{i=1}^N\det{}\!^{-1}
R\Gamma(C,\CO_C(-E_i))\otimes\det R\Gamma(C,\CO_C).\]
In other words, the fiber of $\CP$ at
$D=\sum_{x\in C}\sum_{i=1}^{N-1}\mu_{i,x}\delta_ix+\sum_{x\in C}\sum_{i=1}^N\nu_{i,x}\varepsilon_ix$ is
\begin{equation}
  \label{fibers P}
  \CP_D=\bigotimes_{x\in C}\bigotimes_{i=1}^{N-1}\bomega_x^{-\mu_{i,x}(\mu_{i,x}+1)/2}\otimes
  \bigotimes_{x\in C}\bigotimes_{i=1}^N\bomega_x^{\nu_{i,x}(\nu_{i,x}-1)/2},
\end{equation}
where $\bomega_x$ is the fiber of the canonical line bundle $\bomega_C$ at $x\in C$.

If we choose a decomposition $(\bmu,\bnu)=(\bmu',\bnu')+(\bmu'',\bnu'')$, then we have an
addition of divisors morphism
\[\on{add}\colon C^{(\bmu',\bnu')}\times C^{(\bmu'',\bnu'')}\to C^{(\bmu,\bnu)}.\]
The line bundle $\CP$ enjoys the {\em factorization} property
\begin{equation}
  \label{factor line}
  \on{add}^*\CP|_{\big(C^{(\bmu',\bnu')}_{\leq(\blambda',\btheta')}\times
  C^{(\bmu'',\bnu'')}_{\leq(\blambda'',\btheta'')}\big)_{\on{disj}}}
  \cong\CP\boxtimes\CP|_{\big(C^{(\bmu',\bnu')}_{\leq(\blambda',\btheta')}\times
    C^{(\bmu'',\bnu'')}_{\leq(\blambda'',\btheta'')}\big)_{\on{disj}}}.
\end{equation}

In order to stress the dependence on the base, sometimes we will denote the line bundle
$\CP$ on $C^{(\bmu,\bnu)}$ by $\CP^{(\bmu,\bnu)}$.

\subsection{Monodromic sheaves}
\label{mon shv}
Let $\bCP$ denote the total space of the line bundle $\CP$ with the zero section removed.
We will consider the category $S\Perv_q(\bCP)$ of perverse sheaves of super vector spaces on
$\bCP$, monodromic with monodromy $q$. We describe the most important (for us) object of this
category. We assume that $(\bmu,\bnu)\in-X_{\on{pos}}$, that is
$\bmu+\bnu=-\alpha=-\sum_{j=1}^{2N-2}a_j\alpha_j,\ a_j\in\BN$.
Then $C^{(\bmu,\bnu)}_{\leq(0,0)}\cong C^\alpha\supset\oC^\alpha$: the open subset formed by all the multiplicity
free divisors (i.e.\ each point has multiplicity either 0 or a simple root).
Comparing~\eqref{mixed} and~\eqref{fibers P}, we see that the restriction of the
line bundle $\CP$ to $\oC^\alpha$ trivializes canonically, that is
$\bCP|_{\overset{\circ}{C}{}^\alpha}\cong\BG_m\times\oC^\alpha$. We denote by $\oCI^\alpha$ the local system
on $\bCP|_{\overset{\circ}{C}{}^\alpha}$ equal to the pullback from the $\BG_m$ factor of the
one-dimensional 
local system with monodromy $q$. Finally, we define $\CI^\alpha\in S\Perv_q(\bCP)$
as the Goresky-MacPherson extension of $\oCI^\alpha$ to the whole of $\bCP|_{C^\alpha}$.

Note a crucial difference with a similar construction in~\cite[\S3.4]{g}, where the one-dimensional
local system has the {\em sign} monodromy around $\oC^\alpha$. This is due to the fact that in our
setting the simple roots are all {\em odd}, while in the classical setting the simple roots are all
{\em even}. Also, the absence of signs in our setting will turn out to be compatible with the
structure of SW zastava whose projection to the configuration space is an {\em isomorphism} over
$\oC^\alpha$.

By construction, we have the following factorization isomorphism for $\alpha,\beta\in X_{\on{pos}}$:
\begin{equation}
  \label{factor IC}
  \on{add}^*\CI^{\alpha+\beta}|_{\left(C^\alpha\times C^\beta\right)_{\on{disj}}}
  \cong\CI^\alpha\boxtimes\CI^\beta|_{\left(C^\alpha\times C^\beta\right)_{\on{disj}}}.
\end{equation}
Here $\boxtimes$ stands for the descent of the external product along the morphism of fiberwise
multiplication in $\bCP$ coming from \eqref{factor line}
(the operation corresponding under Riemann--Hilbert to the external product of
twisted $D$-modules).

\subsection{Factorizable sheaves}
\label{factor shv}
A {\em factorizable sheaf} $\CF$ is a collection of monodromic perverse sheaves
$\CF^{(\bmu,\bnu)}\in S\Perv_q(\bCP^{(\bmu,\bnu)})$ equipped with factorization isomorphisms
\begin{equation}
  \label{factor sheaves}
  \on{add}^*\CF^{(\bmu,\bnu)-\beta}|_{\left(C^{(\bmu,\bnu)}\times C^\beta\right)_{\on{disj}}}
  \cong\left(\CF^{(\bmu,\bnu)}\boxtimes\CI^\beta\right)|_{\left(C^{(\bmu,\bnu)}\times C^\beta\right)_{\on{disj}}},
\end{equation}
where in the definition of $\left(C^{(\bmu,\bnu)}\times C^\beta\right)_{\on{disj}}$ a divisor in $C^\beta$
is additionally required to miss $c\in C$. These isomorphisms should be compatible with the ones
in~\eqref{factor IC} under subdivisions of $\beta$.

We also impose the following finiteness conditions:

(a) $\CF^{(\bmu,\bnu)}\ne0$ only for $(\bmu,\bnu)$ belonging to finitely many cosets of the root
lattice $\BZ^{2N-2}\subset X$.

(b) For each such coset, there is $(\blambda,\btheta)$ such that the support of $\CF^{(\bmu,\bnu)}$
lies in $C^{(\bmu,\bnu)}_{\leq(\blambda,\btheta)}$ for $(\bmu,\bnu)$ in this coset.

(c) There are only finitely many $(\bmu,\bnu)$ such that the singular support of $\CF^{(\bmu,\bnu)}$
contains the conormal to (the fiber of $\bCP$ over) the point $(\bmu,\bnu)\cdot c$ in
$C^{(\bmu,\bnu)}_{\leq(\blambda,\btheta)}$.

The factorizable sheaves with the above finiteness conditions form an abelian category $\on{FS}$
(the morphisms are required to be compatible with the factorization isomorphisms).

One can also allow the marked point $c$ to vary in $C$; moreover, one can allow $n$ distinct
marked points to vary in $\oC^n$. The resulting category $\on{FS}_n$~\cite[\S3]{g} is used to
construct a braided tensor structure on $\on{FS}$ via the nearby cycles functor as the marked
points collide. The following theorem is proved similarly to the main result of~\cite{bfs}.
A conceptual proof is due to J.~Lurie, see the proof of~\cite[Theorem 29.2.3]{gl} in the classical
setup.

\begin{thm}
  \label{bfsl}
  There is a braided tensor equivalence $\Rep_q(\GL(N-1|N))\simeq\on{FS}$.
  In particular, for any pair of signatures $(\blambda,\btheta)$ satisfying condition~\eqref{serg},
  the corresponding irreducible $U_q(\fgl(N-1|N))$-module $V_{\blambda,\btheta}$ goes to
  the irreducible factorizable sheaf $\CF_{\blambda,\btheta}$. \hfill $\Box$
\end{thm}

\section{SW zastava}
In~\cite[\S3.3--3.8, \S4]{sw} Y.~Sakellaridis and J.~Wang have defined and studied the zastava spaces
$\mathcal Y$ for spherical varieties. In this section we specialize their results in our particular
case.

\subsection{Open zastava}
\label{open zas}
Given $X_{\on{pos}}\ni\alpha=\sum_{j=1}^{2N-2}a_j\alpha_j,\ a_j\in\BN$, we consider the moduli
space $\oZ^\alpha$ of the following data:

(a) A vector bundle $\CV$ on $C$ of rank $N-1$;

(b) A complete flag $0\subset\CV_1\subset\ldots\subset\CV_{N-2}\subset\CV_{N-1}=\CV$.
Equivalently, line subbundles $\eta_i\colon \CL_i\hookrightarrow\Lambda^i\CV$ of degrees
$\ell_i,\ 1\leq i\leq N-2$, satisfying Pl\"ucker relations. We also set
$\CL_0=\CO_C$, and $\CL_{N-1}=\Lambda^{N-1}\CV$.

(c) A complete flag $\CV\oplus\CO_C=\CU\supset\CU^1\supset\ldots\supset\CU^{N-1}\supset0$.
Equivalently, surjections
$\xi_i\colon \Lambda^i\CU=\Lambda^i\CV\oplus\Lambda^{i-1}\CV\twoheadrightarrow\CK_i$ to line bundles
of degrees $k_i, \ 1\leq i\leq N-1$, satisfying Pl\"ucker relations.

These data should be subject to the following genericity conditions:

\noindent The compositions $\xi_i\circ\eta_i\colon\CL_i\to\CK_i$ and
$\xi_i\circ\eta_{i-1}\colon\CL_{i-1}\to\CK_i$
are nonzero, so that $\CK_i=\CL_{i-1}(D_{2i-1})=\CL_i(D_{2i})$ for effective divisors $D_{2i-1},D_{2i}$
of degrees $a_{2i-1}, a_{2i}$ for $i=1,\ldots,N-1$. In particular,
$k_i=a_1+a_3+\ldots+a_{2i-1}-a_2-a_4-\ldots-a_{2i-2}$, and
$\ell_i=a_1+a_3+\ldots+a_{2i-1}-a_2-a_4-\ldots-a_{2i}$, where we formally set $a_0=0$.

\medskip

The collection of divisors $D_1,\ldots,D_{2N-2}$ defines a {\em factorization} morphism
$\pi\colon\oZ^\alpha\to C^\alpha$. According to~\cite[Proposition 3.4.1]{sw}, it enjoys the
factorization isomorphisms
\[\oZ^{\alpha+\beta}\times_{C^{\alpha+\beta}}\left(C^\alpha\times C^\beta\right)_{\on{disj}}\cong
\left(\oZ^\alpha\times\oZ^\beta\right)\times_{C^\alpha\times C^\beta}
\left(C^\alpha\times C^\beta\right)_{\on{disj}}.\]

According to~\cite[Theorem 6.3.4]{sw}, $\pi$ is a birational isomorphism (in a sharp contrast to
the case of classical zastava spaces).

\subsection{Compactified zastava}
\label{comp zas}
Modifying~\ref{open zas}(b,c) we obtain a relative compactification $\ol{Z}{}^\alpha\supset\oZ^\alpha$
over $C^\alpha$. Namely, we allow the {\em generalized} Borel structures in~\ref{open zas}(b,c),
that is we allow nonzero, but not necessarily fiberwise injective morphisms
$\eta_i\colon\CL_i\to\Lambda^i\CV$,
and not necessarily surjective morphisms $\xi_i\colon\Lambda^i\CU\to\CK_i$.
(However, $\eta_i$ and $\xi_i$ are still injective, resp.\ surjective, {\em generically}.)

According to~\cite[Theorem 6.3.4]{sw}, the proper morphism $\pi\colon\ol{Z}{}^\alpha\to C^\alpha$
is stratified semismall.

We will also need a version
$\pi\colon\ol{Z}{}^{(\bmu,\bnu)}_{\leq(\blambda,\btheta)}\to C^{(\bmu,\bnu)}_{\leq(\blambda,\btheta)}$ with
poles at $c\in C$. It is the moduli space of the following data:

a) An $X$-colored divisor
$D=-\sum_{i=1}^{N-1}\delta_i\Delta_i+\sum_{i=1}^N\varepsilon_iE_i\in C^{(\bmu,\bnu)}_{\leq(\blambda,\btheta)}$;

b) A vector bundle $\CV$ on $C$ of rank $N-1$;

c) A generalized $B_{N-1}$-structure
\[\eta_i\colon\CL_i:=\CO_C(-\Delta_1-\ldots-\Delta_i)\to\Lambda^i\CV,\ 1\leq i\leq N-1,\]

(and $\eta_{N-1}$ is assumed to be an isomorphism);

d) A generalized $B^-_N$-structure {\em with pole at $c\in C$}
\[\xi_i\colon\Lambda^i(\CV\oplus\CO_C)\dasharrow\CK_i:=\CO_C(-E_1-\ldots-E_i),\ 1\leq i\leq N,\]

(and $\xi_N$ is an isomorphism on $C\setminus\{c\}$, but may have zero or pole at $c$),

\noindent subject to the genericity condition that the compositions
$\xi_i\circ\eta_i\colon\CL_i\to\CK_i$ and $\xi_i\circ\eta_{i-1}\colon\CL_{i-1}\to\CK_i$ are nonzero.

\subsection{Convolution diagram}
\label{convo}
We will also need another version of SW zastava with poles: a partial resolution
$\br\colon\wh\CW{}^{(\bmu,\bnu)}_{\leq(\blambda,\btheta)}\to\ol{Z}{}^{(\bmu,\bnu)}_{\leq(\blambda,\btheta)}$
that plays a role of convolution diagram.
First recall that the $\GL(N-1,\bO)$-orbits in $\Gr_{\GL_N}$ are indexed by pairs of signatures
$(\blambda,\btheta)$ of lengths $N-1,N$ respectively. The representatives of these orbits are
written down explicitly e.g.~in the proof of~\cite[Lemma 2.3.2]{bft}. A $\GL(N-1,\bO)$-orbit
in $\Gr_{\GL_N}$ will be denoted $\BO_{\blambda,\btheta}$, and its closure will be denoted
$\ol\BO_{\blambda,\btheta}$.

Now $\wh\CW{}^{(\bmu,\bnu)}_{\leq(\blambda,\btheta)}$ is the moduli space of the following data:

a) An $X$-colored divisor
$D=-\sum_{i=1}^{N-1}\delta_i\Delta_i+\sum_{i=1}^N\varepsilon_iE_i\in C^{(\bmu,\bnu)}_{\leq(\blambda,\btheta)}$;

b) A vector bundle $\CV$ on $C$ of rank $N-1$;

c) A vector bundle $\CU$ on $C$ of rank $N$;

d) An isomorphism $\sigma\colon(\CV\oplus\CO_C)|_{C\setminus\{c\}}\iso\CU|_{C\setminus\{c\}}$
with pole of order $\leq(\blambda,\btheta)$ at $c\in C$. In other words, the Hecke transformation
$\sigma$ lies in the $\GL(N-1,\bO)$-orbit closure
$\GL(N-1,\bO)\backslash\ol\BO_{\blambda,\btheta}\subset\GL(N-1,\bO)\backslash\Gr_{\GL_N}$;

e) A generalized $B^-_{N-1}$-structure
\[\eta_i\colon\CL_i:=\CO_C(-\Delta_1-\ldots-\Delta_i)\to\Lambda^i\CV,\ 1\leq i\leq N-1,\]

(and $\eta_{N-1}$ is assumed to be an isomorphism);

f) A generalized $B^-_N$-structure
\[\xi_i\colon\Lambda^i\CU\to\CK_i:=\CO_C(-E_1-\ldots-E_i),\ 1\leq i\leq N,\]

(and $\xi_N$ is assumed to be an isomorphism),

\noindent subject to the genericity condition that the compositions
$\xi_i\circ \Lambda^i\sigma\circ\eta_i\colon\CL_i\dasharrow\CK_i$ and
$\xi_i\circ \Lambda^i\sigma\circ\eta_{i-1}\colon\CL_{i-1}\dasharrow\CK_i$ are nonzero (but may have
poles at $c\in C$).

\medskip

The same argument as in~\cite[Lemma 4.1.2]{sw} proves that the functor
$\wh\CW{}^{(\bmu,\bnu)}_{\leq(\blambda,\btheta)}$ is representable by (the same named) scheme
of finite type. The morphism
$\br\colon\wh\CW{}^{(\bmu,\bnu)}_{\leq(\blambda,\btheta)}\to\ol{Z}{}^{(\bmu,\bnu)}_{\leq(\blambda,\btheta)}$
takes the morphisms $\xi_i$ in~f) above to the composition
$\xi_i\circ\sigma\colon\Lambda^i(\CV\oplus\CO_C)\dasharrow\CK_i$ (with pole at $c\in C$).
The composition $\pi\circ\br$ is denoted by
$\bq\colon\wh\CW{}^{(\bmu,\bnu)}_{\leq(\blambda,\btheta)}\to C^{(\bmu,\bnu)}_{\leq(\blambda,\btheta)}$.
We also have a morphism \[\bp\colon\wh\CW{}^{(\bmu,\bnu)}_{\leq(\blambda,\btheta)}\to
\GL(N-1,\bO)\backslash\ol\BO_{\blambda,\btheta}\subset\GL(N-1,\bO)\backslash\Gr_{\GL_N}\]
that remembers only $\sigma$ in~d) above (restricted to the formal neighbourhood of $c\in C$).
Finally, we have an open subscheme $\jmath\colon\wt\CW{}^{(\bmu,\bnu)}_{\leq(\blambda,\btheta)}
\hookrightarrow\wh\CW{}^{(\bmu,\bnu)}_{\leq(\blambda,\btheta)}$ cut out by the condition that both
$B^-_{N-1}$- and $B^-_N$-structures in e,f) above are genuine, that is, all $\eta_i$ are
embeddings of line subbundles, and all $\xi_i$ are surjective.

If $(\blambda,\btheta)=(0,0)$, and $(\bmu,\bnu)=-\alpha\in-X_{\on{pos}}$, then clearly
$\wh\CW{}^{-\alpha}_{\leq(0,0)}=\ol{Z}{}^\alpha$ and $\wt\CW{}^{-\alpha}_{\leq(0,0)}=\oZ^\alpha$.

\subsection{Factorization}
\label{Factor zas}
The group $H:=\GL_{N-1}$ is block diagonally embedded into $G:=\GL_{N-1}\times\GL_N$.
We fix $B^-:=B^-_{N-1}\times B^-_N\subset G$ in generic position with respect to $H\subset G$.
Namely, we choose a base $e_1,\ldots,e_N$ of $\BC^N$, so that $\BC^{N-1}$ is spanned by
$e_1,\ldots,e_{N-1}$, and $\GL_{N-1}\subset\GL_N$ is $\GL(\BC^{N-1})$. Now $B^-_{N-1}$ preserves
the flag
\[\BC e_{N-1}\subset\BC e_{N-1}\oplus\BC e_{N-2}\subset\ldots\subset\BC e_{N-1}\oplus\cdots\oplus
\BC e_2\subset\BC^{N-1},\] while $B^-_N$ preserves the flag
\[\BC(e_1+e_N)\subset\BC(e_1+e_N)\oplus\BC(e_2+e_N)\subset\ldots\subset
\BC(e_1+e_N)\oplus\cdots\oplus\BC(e_{N-1}+e_N)\subset\BC^N.\]
The unipotent radical of $B^-$ is denoted by $U^-$.
The data of~\S\ref{convo}a-f) define an $H$-structure on a $G$-bundle
$(\CV\oplus\nobreak\CU)|_{C\setminus\{c\}}$ as well as a generically transversal $B^-$-structure on
$\CV\oplus\CU$. More precisely, these structures are transversal on $C\setminus\{c\}\setminus D$.
These transversal structures give rise to a trivialization of
$(\CV\oplus\CU)|_{C\setminus\{c\}\setminus D}$, i.e.\ a point of the Beilinson-Drinfeld Grassmannian
of $G$ on $C$. The factorization property of the Beilinson-Drinfeld Grassmannian implies
the factorization property of
$\bq\colon\wh\CW{}^{(\bmu,\bnu)}_{\leq(\blambda,\btheta)}\to C^{(\bmu,\bnu)}_{\leq(\blambda,\btheta)}$ as
in~\cite[Proposition 3.4.1]{sw}:
\begin{equation}
  \label{factor zas}
  \wh\CW{}^{(\bmu,\bnu)-\beta}_{\leq(\blambda,\btheta)}\times_{C^{(\bmu,\bnu)-\beta}_{\leq(\blambda,\btheta)}}
  \left(C^{(\bmu,\bnu)}_{\leq(\blambda,\btheta)}\times C^\beta\right)_{\on{disj}}
  \cong\left(\wh\CW{}^{(\bmu,\bnu)}_{\leq(\blambda,\btheta)}\times\ol{Z}{}^\beta\right)
  \times_{C^{(\bmu,\bnu)}_{\leq(\blambda,\btheta)}\times C^\beta}
  \left(C^{(\bmu,\bnu)}_{\leq(\blambda,\btheta)}\times C^\beta\right)_{\on{disj}},
\end{equation}
where $\beta\in X_{\on{pos}}$, and in the definition of
$\left(C^{(\bmu,\bnu)}_{\leq(\blambda,\btheta)}\times C^\beta\right)_{\on{disj}}$ a divisor in $C^\beta$
is additionally required to miss $c\in C$.

Recall that $C^{(\bmu,\bnu)}_{\leq(\blambda,\btheta)}\simeq C^\alpha$ (where
$\alpha=(\blambda,\btheta)-(\bmu,\bnu)$). The factorization property implies an isomorphism
\[\wh\CW{}^{(\bmu,\bnu)}_{\leq(\blambda,\btheta)}\supset\bq^{-1}(C\setminus\{c\})^\alpha\cong
\pi^{-1}(C\setminus\{c\})^\alpha\subset\ol{Z}{}^{(\bmu,\bnu)}_{\leq(\blambda,\btheta)}.\]
According to~\cite[Lemma 6.2.1]{sw}, $\ol{Z}{}^{(\bmu,\bnu)}_{\leq(\blambda,\btheta)}$ is irreducible,
and it is likely that $\wh\CW{}^{(\bmu,\bnu)}_{\leq(\blambda,\btheta)}$ is irreducible as well.
Instead of proving its irreducibility we just restrict our attention to its principal
irreducible component.

\begin{defn}
  \label{olW}
  \textup{(0)} If $(\bmu,\bnu)=(\blambda,\btheta)$, then
  $\wh\CW{}^{(\blambda,\btheta)}_{\leq(\blambda,\btheta)}=\wt\CW{}^{(\blambda,\btheta)}_{\leq(\blambda,\btheta)}$ is just
  one point (see~Lemmas~\ref{central fib} and~\ref{empty} below), and we define
  $\CW^{(\blambda,\btheta)}_{\leq(\blambda,\btheta)}=\ol\CW{}^{(\blambda,\btheta)}_{\leq(\blambda,\btheta)}$ as this
  same point.

  \textup{(a)} If $(\bmu,\bnu)<(\blambda,\btheta)$, then we define
  $\ol\CW{}^{(\bmu,\bnu)}_{\leq(\blambda,\btheta)}$ as the closure of
  $\bq^{-1}(C\setminus\{c\})^\alpha$ in $\wh\CW{}^{(\bmu,\bnu)}_{\leq(\blambda,\btheta)}$.

  \textup{(b)} We define an open subscheme
  $\CW^{(\bmu,\bnu)}_{\leq(\blambda,\btheta)}\subset\ol\CW{}^{(\bmu,\bnu)}_{\leq(\blambda,\btheta)}$
  as $\ol\CW{}^{(\bmu,\bnu)}_{\leq(\blambda,\btheta)}\cap\wt\CW{}^{(\bmu,\bnu)}_{\leq(\blambda,\btheta)}$.

  \textup{(c)} We keep the notation $\jmath$ for the open embedding
  $\CW^{(\bmu,\bnu)}_{\leq(\blambda,\btheta)}\hookrightarrow\ol\CW{}^{(\bmu,\bnu)}_{\leq(\blambda,\btheta)}$.
\end{defn}

Note that $\ol\CW{}^{(\bmu,\bnu)}_{\leq(\blambda,\btheta)}$ and $\CW^{(\bmu,\bnu)}_{\leq(\blambda,\btheta)}$
inherit the factorization property~\eqref{factor zas} from
$\wh\CW{}^{(\bmu,\bnu)-\beta}_{\leq(\blambda,\btheta)}$.

We are interested in the central fiber $\wh\sW{}^{(\bmu,\bnu)}_{\leq(\blambda,\btheta)}
:=\bq^{-1}((\bmu,\bnu)\cdot c)\subset\wh\CW{}^{(\bmu,\bnu)}_{\leq(\blambda,\btheta)}$.
We also set \[\wt\sW{}^{(\bmu,\bnu)}_{\leq(\blambda,\btheta)}=\wh\sW{}^{(\bmu,\bnu)}_{\leq(\blambda,\btheta)}\cap
\wt\CW{}^{(\bmu,\bnu)}_{\leq(\blambda,\btheta)},\
\ol\sW{}^{(\bmu,\bnu)}_{\leq(\blambda,\btheta)}=\wh\sW{}^{(\bmu,\bnu)}_{\leq(\blambda,\btheta)}\cap
\ol\CW{}^{(\bmu,\bnu)}_{\leq(\blambda,\btheta)},\
\sW^{(\bmu,\bnu)}_{\leq(\blambda,\btheta)}=\wh\sW{}^{(\bmu,\bnu)}_{\leq(\blambda,\btheta)}\cap
\CW^{(\bmu,\bnu)}_{\leq(\blambda,\btheta)}.\]

In order to describe
this fiber, note that the embedding $H\hookrightarrow G$ induces an embedding
$H(\bF)\hookrightarrow G(\bF)$. Also, the embedding $\GL_N\hookrightarrow G$ induces an
embedding $\ol\BO_{\blambda,\btheta}\subset\Gr_{\GL_N}\hookrightarrow\Gr_G$. We consider the
$H(\bF)$-saturation $\ol\sO_{\blambda,\btheta}$ of $\ol\BO_{\blambda,\btheta}\subset\Gr_G$.
According to the proof of~\cite[Lemma 2.3.2]{bft}, $\ol\sO_{\blambda,\btheta}$ coincides
with the $H(\bF)$-saturation of $\ol\Gr{}^{\blambda^*}_{\GL_{N-1}}\times\ol\Gr{}^{\btheta}_{\GL_N}$,
where for $\blambda=(\lambda_1,\ldots,\lambda_{N-1})$ we set
$\blambda^*=(-\lambda_{N-1},\ldots,-\lambda_1$).
Finally, we consider the semiinfinite $U^-(\bF)$-orbit
$T^{\bmu^*,\bnu}\subset\Gr_G$. The following lemma is proved as~\cite[Lemma 4.3.2]{sw}:

\begin{lem}
  \label{central fib}
  The following reduced schemes are naturally isomorphic:
  $\left(\wt\sW{}^{(\bmu,\bnu)}_{\leq(\blambda,\btheta)}\right)_{\on{red}}
  \cong\left(T^{\bmu^*,\bnu}\cap\ol\sO_{\blambda,\btheta}\right)_{\on{red}}$
  and $\left(\wh\sW{}^{(\bmu,\bnu)}_{\leq(\blambda,\btheta)}\right)_{\on{red}}
  \cong\left(\ol{T}{}^{\bmu^*,\bnu}\cap\ol\sO_{\blambda,\btheta}\right)_{\on{red}}$. \hfill $\Box$
  \end{lem}

\subsection{Smoothness}
\label{Smooth}
The goal of this subsection is the following

\begin{prop}
  \label{smooth}
  The morphism $\bp\circ\jmath\colon\wt\CW{}^{(\bmu,\bnu)}_{\leq(\blambda,\btheta)}\to
  \GL(N-1,\bO)\backslash\ol\BO_{\blambda,\btheta}$ is smooth.
\end{prop}

In order to prove the proposition, following Y.~Sakellaridis and J.~Wang, we consider
the prestack $\scrM_{H\backslash G,\infty\cdot c}$ classifying the following data:

(a) a $G$-bundle $\scrP$ on $C$;

(b) a reduction of $\scrP$ to $H$ on $C\setminus\{c\}$, i.e.\ a section
$s\colon C\setminus\{c\}\to(H\backslash G)\stackrel{G}{\times}\scrP$.

\begin{lem}[J.~Wang]
  \label{jwang}
  The prestack $\scrM_{H\backslash G,\infty\cdot c}$ is an ind-algebraic stack of ind-locally finite
  type.
\end{lem}

\begin{proof}
  Let $\on{Sect}(C\setminus\{c\},(H\backslash G)\stackrel{G}{\times}\scrP)$ denote the fiber
  of $\scrM_{H\backslash G,\infty\cdot c}$ over a fixed $\scrP$.
  We choose a finite dimensional right $G$-module $V$ with a $G$-equivariant closed embedding
  $H\backslash G\hookrightarrow V$. We consider the associated vector bundle
  $\CV=V\stackrel{G}{\times}\scrP$ on $C\times\Spec R$ for a test ring $R$.
  Let $\on{Sect}(C\setminus\{c\},\CV)$ denote the presheaf over $\Spec R$ representing
  sections $(C\setminus\{c\})\times\Spec R'\to\CV$. Then $\on{Sect}(C\setminus\{c\},\CV)$
  sends $R'$ to
  \[\Hom_{\CO_C\otimes R}(\CV^\vee,\CO_C(\infty\cdot c)\otimes R')=\lim_\to\Gamma(C\times\Spec R',
  \CV(n\cdot c)\otimes R').\]
  For each $n\in\BN$, the presheaf $\Gamma(C\times\Spec R',\CV(n\cdot c)\otimes R')$ is
  represented by a scheme locally of finite type, so $\on{Sect}(C\setminus\{c\},\CV)$ is an
  ind-scheme.

  Now $\on{Sect}(C\setminus\{c\},(H\backslash G)\stackrel{G}{\times}\scrP)\hookrightarrow
  \on{Sect}(C\setminus\{c\},\CV)$ is a closed embedding since the kernel of
  $\on{Sym}(V^*)\twoheadrightarrow\BC[H\backslash G]$ is also finitely generated.
  Hence the morphism $\scrM_{H\backslash G,\infty\cdot c}\to\Bun_G(C)$ is ind-representable,
  and the lemma is proved.
\end{proof}

We have a morphism \[p\colon\scrM_{H\backslash G,\infty\cdot c}\to H(\bF)\backslash G(\bF)/G(\bO)=
H(\bF)\backslash\Gr_G\cong\GL(N-1,\bO)\backslash\Gr_{\GL_N}\] by restricting to the formal
neighbourhood $\wh{C}_c$ of $c\in C$.

\begin{lem}
  \label{abrav}
  The morphism $p$ is formally smooth.
\end{lem}

\begin{proof}
  We consider an auxiliary stack $\scrM'_{H\backslash G,\infty\cdot c}$ over $\scrM_{H\backslash G,\infty\cdot c}$
  classifying the following data:

  (a) a $G$-bundle $\scrP_G$ on $C$;

  (b) an $H$-bundle $\scrP_H$ on $C$;

  (c) an isomorphism $\scrP_G|_{C\setminus c}\iso\Ind_H^G\scrP_H|_{C\setminus c}$.

  \noindent We also have a local version $\scrM^{\prime\on{loc}}_{H\backslash G,\infty\cdot c}$ replacing
  the global curve $C$ by the formal disc $\wh{C}_c$ in the above definition. We have a morphism
  $p'\colon \scrM'_{H\backslash G,\infty\cdot c}\to\scrM^{\prime\on{loc}}_{H\backslash G,\infty\cdot c}$ by
  restricting to the formal neighbourhood $\wh{C}_c\subset C$.

We have the following diagram with cartesian squares:
  \[\begin{CD}
  \Bun_H(C) @<<< \scrM'_{H\backslash G,\infty\cdot c} @>>> \scrM_{H\backslash G,\infty\cdot c}\\
    @VV{p''}V @VV{p'}V @VV{p}V\\
    \Bun_H^{\on{loc}}(C) @<<<  \scrM^{\prime\on{loc}}_{H\backslash G,\infty\cdot c} @>s>> H(\bF)\backslash\Gr_G,
  \end{CD}\]
where $\Bun_H^{\on{loc}}(C)$ stands for $\Bun_H(\wh{C}_c)\cong H(\bO)\backslash\on{pt}$. The fiber of $p''$
is the moduli space of $H$-bundles on $C$ equipped with a trivialization at the formal neighbourhood
of $c\in C$. This is a smooth scheme, so we conclude that $p''$ is formally smooth. It follows that
$p'$ is formally smooth as well. Also note that
$\scrM^{\prime\on{loc}}_{H\backslash G,\infty\cdot c}\cong H(\bO)\backslash\Gr_G$, and the fibers of $s$
are all isomorphic to $H(\bO)\backslash H(\bF)\cong\Gr_H$, so that $s$ is formally smooth as well.
Finally, from the formal smoothness of $s$ and $p'$ we deduce the formal smoothness of $p$.
\end{proof}

Now we are ready to prove~Proposition~\ref{smooth}. We denote by
$\scrM_{H\backslash G,\leq(\blambda,\btheta)\cdot c}\subset\scrM_{H\backslash G,\infty\cdot c}$ the closed substack,
the preimage $p^{-1}(\ol\BO_{\blambda,\btheta})$ of
$\ol\BO_{\blambda,\btheta}\subset\GL(N-1,\bO)\backslash\Gr_{\GL_N}$. By~Lemma~\ref{abrav}, the morphism
$p\colon\scrM_{H\backslash G,\leq(\blambda,\btheta)\cdot c}\to\ol\BO_{\blambda,\btheta}$ is smooth.
On the other hand, the argument of~\cite[\S3.5.3]{sw}
(going back at least to~\cite[Theorem 16.2.1]{gn}) shows that
$\wt\CW{}^{(\bmu,\bnu)}_{\leq(\blambda,\btheta)}$ is locally in smooth topology isomorphic to
$\scrM_{H\backslash G,\leq(\blambda,\btheta)\cdot c}$.
This completes the proof of~Proposition~\ref{smooth}. \hfill $\Box$

\begin{defn}
  \label{pcirc}
  For a constructible complex $\CM$ on $\GL(N-1,\bO)\backslash\ol\BO_{\blambda,\btheta}$ we denote
  by $\bp^\circ\CM$ the constructible complex
  $\bp^*\CM[\dim\ol\CW{}^{(\bmu,\bnu)}_{\leq(\blambda,\btheta)}-\dim\ol\BO_{\blambda,\btheta}]$ on
  $\ol\CW{}^{(\bmu,\bnu)}_{\leq(\blambda,\btheta)}$.
  \end{defn}

\subsection{Semismallness}
The goal of this subsection is the following

\begin{prop}
  \label{semismall}
The morphism $\bq\colon\ol\CW{}^{(\bmu,\bnu)}_{\leq(\blambda,\btheta)}\to C^{(\bmu,\bnu)}_{\leq(\blambda,\btheta)}$
  is stratified semismall.
\end{prop}

\begin{proof}
  First we prove the following

\begin{lem}
  \label{empty}
  \textup{(a)} The intersection $T^{\blambda^*,\btheta}\cap\sO_{\blambda,\btheta}$ consists of just one point.

  \textup{(b)} The intersection $T^{\bmu^{\prime*},\bnu'}\cap\sO_{\blambda,\btheta}$ is empty unless
  $(\bmu',\bnu')\leq(\blambda,\btheta)$.
\end{lem}

\begin{proof}
  (a) One can check that the point in question is the following pair of lattices
$(L'_\blambda,L_\btheta)\in\Gr_{\GL_{N-1}}\times\Gr_{\GL_N}$ (see the proof of~\cite[Lemma 2.3.2]{bft}):
\[L'_\blambda=\bO t^{\lambda_1}e_1\oplus\cdots\oplus\bO t^{\lambda_{N-1}}e_{N-1}\subset\bF\otimes\BC^{N-1},\]
\[L_\btheta=\bO t^{-\theta_1}(e_1+e_N)\oplus\cdots\oplus\bO t^{-\theta_{N-1}}(e_{N-1}+e_N)\oplus\bO
t^{-\theta_N}e_N\subset\bF\otimes\BC^N.\]

(b) According to (the proof of)~\cite[Lemma 2.3.2]{bft}, the image of the convolution of
  $\Gr^{\blambda'}_{\GL_{N-1}}$ and $\Gr^{\btheta'}_{\GL_N}$ contains $\BO_{\blambda',\btheta'}$ as a dense open
  subvariety.
  The closure of a $\GL(N-1,\bO)$-orbit $\BO_{\blambda',\btheta'}$ contains $\BO_{\blambda,\btheta}$ if and
  only if $(\blambda',\btheta')\geq(\blambda,\btheta)$. This claim for $\SO(N-1,\bO)$-orbits in
  $\Gr_{\SO_N}$ is proved in~\cite[Theorem 3.3.5(a,b)]{bft}. The proof for $\GL(N-1,\bO)$-orbits in
  $\Gr_{\GL_N}$ is absolutely similar. It follows that
  $\Gr^{\blambda^{\prime*}}_{\GL_{N-1}}\times\Gr^{\btheta'}_{\GL_N}$ does not intersect $\sO_{\blambda,\btheta}$
  unless $(\blambda',\btheta')\geq(\blambda,\btheta)$.

  More generally, we consider the following triple convolution fiber. Given signatures
  $\bxi,\bzeta$ of length $N-1$ and $\bta,\brho$ of length $N$, we consider the lattices
  $L'_\bzeta$ and $L_\bta$ as in the proof of (a) and the moduli space $\sM$ of quadruples
  $(L'_\bzeta,L'_2,L_3,L_\bta)$, where $L'_2$ (resp.\ $L_3$) is a lattice in $\bF\otimes\BC^{N-1}$
  (resp.\ in $\bF\otimes\BC^N$), and the relative position of $(L'_\bzeta,L'_2)$ is $\bxi$,
  the relative position of $(L'_2,L_3)$ is $(\blambda,\btheta)$, while the relative position
  of $(L_3,L_\bta)$ is $\brho$. Then $\sM$ is empty unless
  $(\bxi,\brho)+(\blambda,\btheta)\geq(\bzeta,\bta)$.

  On the other hand, $\sM$ is equal to the intersection
  $({}^{\bzeta^*}\!\Gr_{\GL_{N-1}}^{\bxi^*}\times{}^\bta\Gr_{\GL_N}^\brho)\cap\sO_{\blambda,\btheta}$,
  where $^{\bzeta^*}\!\Gr_{\GL_{N-1}}^{\bxi^*}$ stands for the moduli space of lattices in $\bF\otimes\BC^{N-1}$
  in relative position $\bxi^*$ with respect to $L'_{\bzeta^*}$, and $^\bta\Gr_{\GL_N}^\brho$
stands for the moduli space of lattices in $\bF\otimes\BC^N$
in relative position $\brho$ with respect to $L_\bta$.

If both $\bzeta$ and $\bxi$ tend to infinity in the cone of dominant weights of $\GL_{N-1}$ so
that the difference $\bzeta-\bxi$ is fixed to be equal to $\bmu'$, then
$^{\bzeta^*}\!\Gr_{\GL_{N-1}}^{\bxi^*}$ tends to $T^{\bmu^{\prime*}}$.
Similarly, if both $\bta$ and $\brho$ tend to infinity in the cone of dominant weights of $\GL_N$ so
that the difference $\bta-\brho$ is fixed to be equal to $\bnu'$, then
$^\bta\Gr_{\GL_N}^\brho$ tends to $T^{\bnu'}$. This proves (b).
  \end{proof}

We return to the proof of~Proposition~\ref{semismall}.
By the factorization property~\eqref{factor zas},
  $\dim\ol\CW{}^{(\bmu,\bnu)}_{\leq(\blambda,\btheta)}=\dim\ol{Z}{}^\alpha=|\alpha|$ (the second equality
  follows from~\cite[Theorem 6.3.4]{sw}), where
  $\alpha=\sum_{j=1}^{2N-2}a_j\alpha_j:=(\blambda,\btheta)-(\bmu,\bnu)$,
  and $|\alpha|:=\sum_{j=1}^{2N-2}a_j$. Again by the factorization property, it suffices to
  check $\dim\ol\sW{}^{(\bmu,\bnu)}_{\leq(\blambda,\btheta)}\leq\dim\wh\sW{}^{(\bmu,\bnu)}_{\leq(\blambda,\btheta)}
  \leq|\alpha|/2$. We repeat the argument of the
  proof of~\cite[Proposition 6.1.1]{sw} based on the fact that the boundary
  $\partial\ol{T}{}^{\bmu^*,\bnu}=\ol{T}{}^{\bmu^*,\bnu}\setminus T^{\bmu^*,\bnu}$ is a Cartier divisor
  in $\ol{T}{}^{\bmu^*,\bnu}$. We have
  $\partial\ol{T}{}^{\bmu^*,\bnu}=\bigsqcup_{(\bmu',\bnu'){\underset{G}{>}}(\bmu,\bnu)}T^{\bmu^{\prime*},\bnu'}$.
  Here $(\bmu',\bnu')\underset{G}{>}(\bmu,\bnu)$
means that $0\ne(\bmu',\bnu')-(\bmu,\bnu)$ is a nonnegative linear combination of simple
roots of $G$. These simple roots are
  the following linear combinations of $\alpha_1,\ldots,\alpha_{2N-2}$:
  \begin{multline*}
\beta_1=\alpha_1+\alpha_2,\ \beta_2=\alpha_3+\alpha_4,\ldots,\beta_{N-1}=\alpha_{2N-3}+\alpha_{2N-2};\\
\gamma_1=\alpha_2+\alpha_3,\ \gamma_2=\alpha_4+\alpha_5,\ldots,\gamma_{N-2}=\alpha_{2N-4}+\alpha_{2N-3}.
  \end{multline*}
It follows from~Lemma~\ref{empty} that
$\dim\left(\ol{T}{}^{\bmu^*,\bnu}\cap\ol\sO_{\blambda,\btheta}\right)\leq
\max\{\sum_{i=1}^{N-1}b_i+\sum_{k=1}^{N-2}c_k\}$, where the maximum is taken over the set
\[\{(\bmu',\bnu')\in X\ :\ (\bmu,\bnu)\underset{G}{\leq}(\bmu',\bnu')\leq(\blambda,\btheta)\},\]
and $(\bmu',\bnu')-(\bmu,\bnu)=\sum_{i=1}^{N-1}b_i\beta_i+\sum_{k=1}^{N-2}c_k\gamma_k$.
Now since each simple root of $G$ is a sum of {\em two} simple roots of $\GL(N-1|N)$,
the above maximum is at most $|\alpha|/2$.
\end{proof}

\subsection{Two line bundles}
\label{Two line}
Let $\CalD$ denote the determinant line bundle on $\Gr_{\GL_N}$. It carries a canonical
$\GL(N-1,\bO)$-equivariant structure that defines the same named line bundle on the quotient
stack $\GL(N-1,\bO)\backslash\Gr_{\GL_N}$.





Recall the morphism \[\bp\colon\wh\CW{}^{(\bmu,\bnu)}_{\leq(\blambda,\btheta)}\to
\GL(N-1,\bO)\backslash\ol\BO_{\blambda,\btheta}\subset\GL(N-1,\bO)\backslash\Gr_{\GL_N}\]
defined in~\S\ref{convo}.

Thus we can consider the line bundle $\bp^*\CalD$ on $\wh\CW{}^{(\bmu,\bnu)}_{\leq(\blambda,\btheta)}$.
We also have a line bundle $\bq^*\CP$ (see~\S\ref{factor bundle}) on
$\wh\CW{}^{(\bmu,\bnu)}_{\leq(\blambda,\btheta)}$.

\begin{lem}
  \label{compare line bun}
  We have a canonical isomorphism $\jmath^*\bp^*\CalD\cong\jmath^*\bq^*\CP$ of line
  bundles on $\wt\CW{}^{(\bmu,\bnu)}_{\leq(\blambda,\btheta)}$.
\end{lem}

\begin{proof}
  The fiber of $\CalD$ at a point $(\CV,\CU,\sigma)\in\GL(N-1,\bO)\backslash\Gr_{\GL_N}$ is
  equal to $\det R\Gamma(C,\CV)\otimes\det R\Gamma(C,\CO_C)\otimes\det^{-1}R\Gamma(C,\CU)$.
  Now compare with the definition of the line bundle $\CP$ in~\S\ref{factor bundle}.
\end{proof}

Now let $Y$ be a top-dimensional irreducible component of $\sW^{(\bmu,\bnu)}_{\leq(\blambda,\btheta)}$
(that is $\dim Y=\frac12|(\blambda,\btheta)-(\bmu,\bnu)|$),
and let $\ol{Y}$ denote its closure in $\ol\sW{}^{(\bmu,\bnu)}_{\leq(\blambda,\btheta)}$. Then the
boundary $\partial\ol{Y}=\ol{Y}\setminus Y$ is a Cartier divisor in $\ol{Y}$, a union of irreducible
Weil divisors $\partial\ol{Y}=\bigcup_{i=1}^n\partial_i\ol{Y}$. According to~Lemma~\ref{compare line bun},
the line bundle $\bp^*\CalD$ restricted to $\ol\sW{}^{(\bmu,\bnu)}_{\leq(\blambda,\btheta)}$ has a canonical
(up to scalar multiplication) nonvanishing section $s$ on the open subscheme
$\sW^{(\bmu,\bnu)}_{\leq(\blambda,\btheta)}\subset\ol\sW{}^{(\bmu,\bnu)}_{\leq(\blambda,\btheta)}$.

\begin{lem}
  \label{pole}
  There is $i$ such that $s|_Y$ viewed as a rational section on $\ol{Y}$,
  has a zero or pole at $\partial_i\ol{Y}$.
\end{lem}

\begin{proof}
  Recall that
  $\ol{Y}\subset\ol\sW{}^{(\bmu,\bnu)}_{\leq(\blambda,\btheta)}\subset\Gr_G=\Gr_{\GL_{N-1}}\times\Gr_{\GL_N}$.
  Accordingly, the boundary components $\partial_i\ol{Y}$ can be of the following 3 types:

  (a) $\partial_i\ol{Y}$ is an irreducible component of the intersection
  $\ol{T}{}^{\bmu^*+\gamma_k,\bnu}\cap\ol\sO_{\blambda,\btheta}$ for some simple root $\gamma_k$ of
  $\GL_{N-1},\ 1\leq k\leq N-2$, but {\em not} an irreducible component of the intersection
  $\ol{T}{}^{\bmu^*,\bnu+\beta_l}\cap\ol\sO_{\blambda,\btheta}$ for any simple root $\beta_l$ of
  $\GL_N,\ 1\leq l\leq N-1$.

  (b) $\partial_i\ol{Y}$ is an irreducible component of the intersection
  $\ol{T}{}^{\bmu^*,\bnu+\beta_l}\cap\ol\sO_{\blambda,\btheta}$ for some simple root $\beta_l$ of
  $\GL_N,\ 1\leq l\leq N-1$, but {\em not} an irreducible component of the intersection
$\ol{T}{}^{\bmu^*+\gamma_k,\bnu}\cap\ol\sO_{\blambda,\btheta}$ for any simple root $\gamma_k$ of
  $\GL_{N-1},\ 1\leq k\leq N-2$.

  (c) $\partial_i\ol{Y}$ is an irreducible component of the intersection
  $\ol{T}{}^{\bmu^*+\gamma_k,\bnu}\cap\ol\sO_{\blambda,\btheta}$ for some simple root $\gamma_k$ of
  $\GL_{N-1},\ 1\leq k\leq N-2$, and also {\em is} an irreducible component of the intersection
  $\ol{T}{}^{\bmu^*,\bnu+\beta_l}\cap\ol\sO_{\blambda,\btheta}$ for some simple root $\beta_l$ of
  $\GL_N,\ 1\leq l\leq N-1$.

  Actually, the case (c) never happens. Otherwise, $\partial_i\ol{Y}$ is an irreducible
  component of the intersection $\ol{T}{}^{\bmu^*+\gamma_k,\bnu+\beta_l}\cap\ol\sO_{\blambda,\btheta}$.
  This intersection has dimension at most $\frac12|(\blambda,\btheta)-(\bmu,\bnu)|-2$ by
  the proof of~Proposition~\ref{semismall} (since $|\gamma_k|=|\beta_l|=2$). On the other
  hand, $\dim\partial_i\ol{Y}=\dim\ol{Y}-1=\frac12|(\blambda,\btheta)-(\bmu,\bnu)|-1$: a
  contradiction.

  Now the line bundle $\bp^*\CalD$ restricted to $\ol{Y}$, by definition, is the ratio of
  the very ample determinant line bundle $\CalD_N$ on $\Gr_{\GL_N}$ and the very ample
  determinant line bundle $\CalD_{N-1}$ on $\Gr_{\GL_{N-1}}$ (we view
  $\ol{Y}\subset\ol\sW{}^{(\bmu,\bnu)}_{\leq(\blambda,\btheta)}$ as a closed subscheme of
  $\Gr_{\GL_{N-1}}\times\Gr_{\GL_N}$). The section $s|_Y$ is the ratio of the trivialization
  of $\CalD_N$ on $T^\bnu$ and the trivialization of $\CalD_{N-1}$ on $T^{\bmu^*}$.
  Hence the section $s|_Y$ viewed as a rational section
  on $\ol{Y}$ must have a pole at some boundary component of type~(a) (unless $N=2$, and there
  are no type~(a) components whatsoever), and also a zero at some boundary component of type~(b).
  We already know that there are no cancellations between components of types~(a) and~(b) (that
  is, there are no components of type~(c)). This completes the proof of the lemma.
\end{proof}

\subsection{Multiple marked points}
As in the end of~\S\ref{factor shv}, one can also allow the marked point $c$ to vary in $C$;
moreover, one can allow $n$ distinct marked points to vary in $\oC^n$.
One obtains the SW zastava spaces with poles at the marked points, e.g.\
$\bq\colon \ol\CW{}^{(\bmu,\bnu)}_{\leq(\blambda^{(1)},\btheta^{(1)}),\ldots,(\blambda^{(n)},\btheta^{(n)})}\to
C^{(\bmu,\bnu)}_{\leq(\blambda^{(1)},\btheta^{(1)}),\ldots,(\blambda^{(n)},\btheta^{(n)})}\times\oC^n$.
The definition is similar to~Definition~\ref{olW} and is left to the reader.
The obvious analogues of the results of~\S\S\ref{Factor zas}--\ref{Two line} hold true with
similar proofs, e.g.\ the morphism $\bq$ above is semismall.

\section{The Gaiotto category and the functor to factorizable sheaves}

    \subsection{Classification of irreducible monodromic perverse sheaves}
Recall that $\CalD$ denotes the determinant line bundle on $\Gr_{\GL_N}$, and let $\bCD$ denote
    the punctured total space of $\CalD$. We consider the equivariant derived constructible category
    $SD^b_{\GL(N-1,\bO),q}(\bCD)$ of sheaves of super vector spaces on $\bCD$, monodromic
    with monodromy $q$. Recall that $q$ is assumed to be a transcendental complex number.
    We also consider the abelian
    category $S\Perv_{\GL(N-1,\bO),q}(\bCD)$ of perverse sheaves of super vector spaces
    on $\bCD$, monodromic with monodromy $q$.

    Recall that the $\GL(N-1,\bO)$-orbits on $\Gr_{\GL_N}$ are numbered by pairs of signatures
    $(\blambda,\btheta)$ such that the length of $\blambda$ (resp.\ $\btheta$) is $N-1$ (resp.\ $N$).
    We first address the question when the closure of an orbit $\BO_{\blambda,\btheta}$ supports an irreducible
    $\GL(N-1,\bO)$-equivariant $q$-monodromic perverse sheaf.

    \begin{prop}
      \label{supports}
      The closure of an orbit $\BO_{\blambda,\btheta}$ supports an irreducible
      $\GL(N-1,\bO)$-equivariant $q$-monodromic perverse sheaf iff $(\blambda,\btheta)$ satisfies the
      condition~\eqref{serg}.
    \end{prop}

    \begin{proof}
      Recall the point $L_{\blambda,\btheta}\in\BO_{\blambda,\btheta}$ (see the proof
      of~\cite[Lemma~2.3.2]{bft}):
      \[L_{\blambda,\btheta}=\bO(t^{-\lambda_1-\theta_1}e_1+t^{-\theta_1}e_N)\oplus\ldots\oplus
      \bO(t^{-\lambda_{N-1}-\theta_{N-1}}e_{N-1}+t^{-\theta_{N-1}}e_N)\oplus\bO t^{-\theta_N}e_N.\]
      The closure of the orbit $\BO_{\blambda,\btheta}$ supports an irreducible
      $\GL(N-1,\bO)$-equivariant $q$-monodromic perverse sheaf iff the stabilizer of
$L_{\blambda,\btheta}$ in $\GL(N-1,\bO)$ acts trivially in the fiber of $\CalD$ over $L_{\blambda,\btheta}$.
This stabilizer (more precisely, its reductive part) is computed in the proof
of~\cite[Lemma~2.3.3]{bft}.

Assume for example that $\lambda_{i-1}+\theta_i>\theta_i+\lambda_i=\lambda_i+\theta_{i+1}=\ldots
=\lambda_{j-1}+\theta_j>\theta_j+\lambda_j$, the value of the sums in the middle (equal to each other)
is $a$, and the number of such sums is $n\geq 2$. We set $m:=\lfloor\frac{n}{2}\rfloor$.
Then the reductive part of $\on{Stab}_{\GL(N-1,\bO)}(L_{\blambda,\btheta})$ has a factor $\GL_m$.
One can check that the character of the action of $\on{Stab}_{\GL(N-1,\bO)}(L_{\blambda,\btheta})$
on the fiber $\CalD_{L_{\blambda,\btheta}}$ restricted to the factor $\GL_m$ equals $\det^a$.
    \end{proof}

\begin{defn}
\label{typic}
\textup{(a)} We say that a bisignature $(\blambda,\btheta)$ satisfying the condition~\eqref{serg}
is {\em relevant}. The corresponding orbit $\BO_{\blambda,\btheta}$ will be called {\em relevant} as well.

\textup{(b)} We say that a bisignature $(\blambda,\btheta)$ is {\em typical} if it satisfies
the following condition: $\lambda_i+\theta_j\ne0$ for any $i,j$.
\end{defn}

\subsection{The functor $F$}
Recall the setup of~\S\ref{Factor zas}. For $(\blambda',\btheta')\geq(\blambda,\btheta)$
we have an evident closed embedding
$\ol\CW{}^{(\bmu,\bnu)}_{\leq(\blambda,\btheta)}\hookrightarrow\ol\CW{}^{(\bmu,\bnu)}_{\leq(\blambda',\btheta')}$
compatible with the closed embedding
$C^{(\bmu,\bnu)}_{\leq(\blambda,\btheta)}\hookrightarrow C^{(\bmu,\bnu)}_{\leq(\blambda',\btheta')}$.

Given $\CM\in SD^b_{\GL(N-1,\bO),q}(\bCD)$ supported at $\ol\BO_{\blambda,\btheta}$ we define a
factorizable complex $\CF=(\CF^{(\bmu,\bnu)})=F(\CM)\in D(\on{FS})$ as follows.
Due to~Lemma~\ref{compare line bun}, for any $(\bmu,\bnu)\in X,\ \jmath^*\bp^\circ\CM$
(see~Definition~\ref{pcirc}) is a
$q$-monodromic complex on the punctured line bundle $\jmath^*\bq^*\bCP$ on
$\CW^{(\bmu,\bnu)}_{\leq(\blambda,\btheta)}$. Hence $\jmath_!\jmath^*\bp^\circ\CM$ is a $q$-monodromic
complex on the punctured line bundle $\bq^*\bCP$ on $\ol\CW{}^{(\bmu,\bnu)}_{\leq(\blambda,\btheta)}$,
and $\bq_*\jmath_!\jmath^*\bp^\circ\CM$ is a $q$-monodromic complex on the punctured line bundle
$\bCP$ on $C^{(\bmu,\bnu)}_{\leq(\blambda,\btheta)}$. We set
$\CF^{(\bmu,\bnu)}=\bq_*\jmath_!\jmath^*\bp^\circ\CM$. Due to the previous paragraph,
$\CF^{(\bmu,\bnu)}$ is well defined (is independent of the choice of $(\blambda,\btheta)$
such that $\ol\BO_{\blambda,\btheta}$ contains the support of $\CM$).

We will show that if $\CM$ is perverse, then all $\CF^{(\bmu,\bnu)}$ are perverse as well, and they
form a factorizable sheaf. We start with the following cleanness result.

\begin{lem}
  \label{clean}
  The natural morphism $\jmath_!\jmath^*\bp^\circ\CM\to\jmath_*\jmath^*\bp^\circ\CM$ of $q$-monodromic
  complexes on the punctured line bundle $\bq^*\bCP$ on $\ol\CW{}^{(\bmu,\bnu)}_{\leq(\blambda,\btheta)}$
  is an isomorphism.
\end{lem}

\begin{proof}
The argument is the same as the proof of~\cite[Theorem 7.3]{g}. Namely, the argument
in~\cite[\S\S7.8,7.10]{g} 
reduces the cleanness property in question to the cleanness property~\cite[Theorem 7.6]{g}
of $\ol\Bun_{B^-}(C)$, where $B^-=B_{N-1}^-\times B_N^-\subset G=\GL_{N-1}\times\GL_N$.
\end{proof}

\begin{lem}
  \label{irreduc}
  For $\CM=\IC_{\blambda,\btheta}^q$ the corresponding $q$-monodromic complex $\CF^{(\bmu,\bnu)}$ is
  an irreducible perverse sheaf for any $(\bmu,\bnu)\leq(\blambda,\btheta)$.
\end{lem}

\begin{proof}
  The perverseness of $\CF^{(\bmu,\bnu)}$ follows from the smoothness of $\bp\circ\jmath$
  (Proposition~\ref{smooth}) and the semismallness of $\bq$~(Proposition~\ref{semismall})
  along with the cleanness result of~Lemma~\ref{clean}.
  The irreducibility of $\CF^{(\bmu,\bnu)}$ follows from the factorization property and
  the vanishing of top degree compactly supported cohomology of the central fiber
  $\sW^{(\bmu,\bnu)}_{\leq(\blambda,\btheta)}$ with coefficients in the restriction of
  $\jmath^*\bp^\circ\CM$ to the canonical (up to scalar multiplication) section $s$ of
  $\bp^*\bCD$ on $\sW^{(\bmu,\bnu)}_{\leq(\blambda,\btheta)}$ (see~\S\ref{Two line}).
  Here `top degree' means $|(\blambda,\btheta)-(\bmu,\bnu)|$, see the proof
  of~Proposition~\ref{semismall}. Since the dimension of the central fiber is at most
  $\frac12|(\blambda,\btheta)-(\bmu,\bnu)|$, the desired cohomology vanishing follows if we
  know that the monodromy of the local system $s^*\jmath^*\bp^\circ\CM$ on a nonempty open subvariety
  $Y^0$ of each top-dimensional irreducible component $Y$ of the central fiber is {\em nontrivial}.
  This nontriviality follows in turn from~Lemma~\ref{pole} and the fact that $q$ has infinite order
  in $\BC^\times$. Indeed, the monodromy of the local system around the boundary component
  $\partial_i\overline{Y}\subset\overline{Y}$ equals $q$ raised to the power equal to the order of
  the zero or pole of $s$ at $\partial_i\overline{Y}$.
\end{proof}

\begin{cor}
  \label{irreducib}
  \textup{(a)} $F$ is an exact functor $S\Perv_{\GL(N-1,\bO),q}(\bCD)\to\on{FS}$.

  \textup{(b)} We have $F(\IC^q_{\blambda,\btheta})=\CF_{\blambda,\btheta}$.

  \textup{(c)} $F$ is conservative and faithful.
\end{cor}

\begin{proof}
  (a) The exactness follows just as in the beginning of the proof of~Lemma~\ref{irreduc}.
  The factorization
  property of $F(\CM)$ follows from the factorization property of SW zastava with poles
  in~\S\ref{Factor zas} and the isomorphism $\CF^{(\bmu,\bnu)}\cong\CI^\alpha$ for
  $\CM=\IC^q_{0,0},\ \alpha\in X_{\on{pos}}$ and $(\bmu,\bnu)=-\alpha$ (notation of~\S\ref{mon shv}).
  The latter isomorphism is tautological on $\oC^\alpha$ since
  $\bq=\pi\colon\ol\CW{}^{(\bmu,\bnu)}_{\leq(0,0)}=\ol{Z}{}^\alpha\to C^\alpha$
  is an isomorphism over $\oC^\alpha$ (notation of~\S\ref{comp zas}). It extends to the whole
  of $C^\alpha$ due to the irreducibility property of~Lemma~\ref{irreduc}.
  The finiteness properties~\ref{factor shv}(a,b) are evident, and~\ref{factor shv}(c)
  follows since it is satisfied for an irreducible factorizable sheaf $\CF_{\blambda,\btheta}$
  (for any $(\blambda,\btheta)$ satisfying~\eqref{serg}) by~Theorem~\ref{bfsl}.

  (b) is immediate from~Lemma~\ref{irreduc}, and~(c) follows from~(a) and~(b).
\end{proof}

Recall that the abelian category $S\Perv_{\GL(N-1,\bO),q}(\bCD)$ is equipped with a braided tensor
structure given by {\em fusion} $\star$ as in~\cite[\S4.1]{bfgt}. Note that the associativity
of the fusion monoidal structure is not obvious at all (we need to construct an isomorphism
between various iterated nearby cycles). However, it follows from the next proposition ($F$ is
a monoidal functor to the category $\on{FS}$ with associative monoidal structure)
and~Corollary~\ref{irreducib} ($F$ is conservative).

\begin{prop}
  \label{braided}
$F$ is a braided tensor functor inducing an isomorphism of Grothendieck rings of
$S\Perv_{\GL(N-1,\bO),q}(\bCD)$ and $\on{FS}$.
\end{prop}

\begin{proof}
  The braided tensor structures on both categories $S\Perv_{\GL(N-1,\bO),q}(\bCD)$ and $\on{FS}$ are
  defined via nearby cycles. The nearby cycles commute with the functor $F$ because
  $\bp\circ\jmath$ is smooth, while $\bq$ is proper (we are using the cleanness property
  of~Lemma~\ref{clean}). The fact that $F$ induces an isomorphism of Grothendieck rings follows
  from~Corollary~\ref{irreducib}.
\end{proof}

\subsection{Rigidity}

We denote by $\IC^q_{\on{taut}}$ the irreducible sheaf $\IC^q_{(0,\ldots,0),(1,0,\ldots,0)}$.
We denote by $(\IC^q_{\on{taut}})^*$ the irreducible sheaf $\IC^q_{(0,\ldots,0),(0,\ldots,0,-1)}$.
We denote by $\IC^q_{\on{ad}}$ the irreducible sheaf $\IC^q_{(0,\ldots,0),(1,0,\ldots,0,-1)}$.

\begin{lem}
  \label{rig taut}
  We have $\IC^q_{\on{taut}}\star(\IC^q_{\on{taut}})^*\simeq\IC^q_{\on{ad}}\oplus\IC^q_{0,0}$.
\end{lem}

\begin{proof}
  The class of $\IC^q_{\on{taut}}\star(\IC^q_{\on{taut}})^*$ in $K(S\Perv_{\GL(N-1,\bO),q}(\bCD))$ equals
  $[\IC^q_{\on{ad}}]+[\IC^q_{0,0}]$ by~Proposition~\ref{braided},~Corollary~\ref{irreducib}
  and~Theorem~\ref{bfsl}.
  It remains to check that $\IC^q_{\on{taut}}\star(\IC^q_{\on{taut}})^*$ is semisimple. This follows from
  the next observation. The sheaf $\IC^q_{0,0}$ is equivariant with respect to the loop rotation,
  while $\IC^q_{\on{ad}}$ is monodromic with respect to the loop rotation with monodromy $q^2$.
  In fact, $\IC^q_{\bmu,\bnu}$ is monodromic with respect to the loop rotation with monodromy
  $q^{(\bnu,\bnu)-(\bmu,\bmu)}$. Indeed, the point $L_{\bmu,\bnu}\in\BO_{\bmu,\bnu}$ (see the proof
  of~\cite[Lemma 2.3.2]{bft}) is not invariant with respect to the loop rotation group
  $\BC^\times_{\on{rot}}$, but is invariant with respect to an appropriate one-parametric subgroup
  $\BC^\times\hookrightarrow T\times\BC^\times_{\on{rot}}$ (here $T\subset\GL_N$ is the diagonal Cartan torus)
  whose projection to $\BC^\times_{\on{rot}}$ is an isomorphism. Now it is straightforward to compute
  the monodromy in the fiber of $\bCD$ over $L_{\bmu,\bnu}$ of the above one-parametric subgroup.
\end{proof}

We consider the full subcategory $\CE$ of $S\Perv_{\GL(N-1,\bO),q}(\bCD)$ fusion generated by
$\IC^q_{\on{taut}},(\IC^q_{\on{taut}})^*,\IC^q_{0,0}$: it is the smallest full abelian subcategory
containing the above sheaves, closed under taking fusion products, images, kernels and cokernels.

\begin{lem}
  \label{etingof}
  $F|_\CE\colon \CE\to\on{FS}$ is a braided tensor equivalence.
\end{lem}

\begin{proof} (P.~Etingof)
  Since $F$ is injective on morphisms, we have to check that $F|_\CE$ is surjective on morphisms
  and essentially surjective. The irreducibles
  $F(\IC^q_{0,0})\simeq\CF_{0,0},\ F(\IC^q_{\on{taut}})\simeq\CF_{\on{taut}},\
  F(\IC^q_{\on{taut}})^*\simeq\CF_{\on{taut}}^*$ correspond under the braided tensor equivalence
  $\on{FS}\cong\Rep_q(\GL(N-1|N))$ of~Theorem~\ref{bfsl} to the tensor unit~$\BC$,
  the tautological representation $V=\BC^{N-1|N}$ and its dual $V^*$. According to~Lemma~\ref{rig taut},
  the morphisms $\CF_{0,0}\to\CF_{\on{taut}}\star\CF_{\on{taut}}^*$ and
  $\CF_{\on{taut}}^*\star\CF_{\on{taut}}\to\CF_{0,0}$ corresponding to the rigidity morphisms
  $\BC\to V\otimes V^*,\ V^*\otimes V\to\BC$, lie in the image of $\Hom_\CE$.
  However, the above objects and morphisms
  generate $\on{FS}\cong\Rep_q(\GL(N-1|N))$ in the sense that

  a) the vector space
  $\Hom_{\Rep_q(\GL(N-1|N))}(V^{\otimes n}\otimes V^{*\otimes m},V^{\otimes n'}\otimes V^{*\otimes m'})$
  is spanned by the tangle diagrams for any $m,n,m',n'\in\BN$. Hence
  \begin{multline*}
    \Hom_{\on{FS}}(\CF_{\on{taut}}^{\star n}\star(\CF_{\on{taut}}^*)^{\star m},
  \CF_{\on{taut}}^{\star n'}\star(\CF_{\on{taut}}^*)^{\star m'})\\
  \simeq\Hom_{\Rep_q(\GL(N-1|N))}(V^{\otimes n}\otimes V^{*\otimes m},V^{\otimes n'}\otimes V^{*\otimes m'})
  \end{multline*}
  is isomorphic to
  $\Hom_\CE\left((\IC^q_{\on{taut}})^{\star n}\star((\IC^q_{\on{taut}})^*)^{\star m},
  (\IC^q_{\on{taut}})^{\star n'}\star((\IC^q_{\on{taut}})^*)^{\star m'}\right)$;

  b) every object of $\Rep_q(\GL(N-1|N))$ is
  a subquotient of a tensor product $V^{\otimes n}\otimes V^{*\otimes m}$ for some $m,n\in\BN$.
  In particular, every projective-injective object of the Frobenius category
  $\on{FS}\cong\Rep_q(\GL(N-1|N))$ is a direct summand of
  $\CF_{\on{taut}}^{\star n}\star(\CF_{\on{taut}}^*)^{\star m}$ for some $m,n\in\BN$.

  For a proof for $\Rep(\fgl(M|N))$ see e.g.~\cite[Lemma 1.4.4(i)]{c}; the proof for
  $\Rep_q(\GL(N-1|N))$ for transcendental $q$ is the same.

  Hence every object of $\on{FS}$ is isomorphic to $F(\CM)$ for some $\CM\in\CE$,
  being the image of a morphism from a projective
  object (= a direct summand of $\CF_{\on{taut}}^{\star n}\star(\CF_{\on{taut}}^*)^{\star m}$) to an injective
  object (= a direct summand of $\CF_{\on{taut}}^{\star n'}\star(\CF_{\on{taut}}^*)^{\star m'}$). And every
  morphism $F(\CM_1)\to F(\CM_2)$ is in the image of $\Hom_\CE(\CM_1,\CM_2)$ as a morphism
  from a quotient $F(\CM_1)$ of $\CF_{\on{taut}}^{\star n}\star(\CF_{\on{taut}}^*)^{\star m}$ to a subsheaf
  $F(\CM_2)$ of $\CF_{\on{taut}}^{\star n'}\star(\CF_{\on{taut}}^*)^{\star m'}$.

  Thus $F|_\CE\colon\CE\iso\on{FS}$ is a braided tensor equivalence.
\end{proof}

\begin{cor}
  \label{rig}
  All the irreducibles $\IC^q_{\bmu,\bnu}$ are rigid (in the sense of the fusion tensor structure
  $\star$ on $S\Perv_{\GL(N-1,\bO),q}(\bCD)$).
\end{cor}

\begin{proof}
  The category $\on{FS}\cong\Rep_q(\GL(N-1|N))$ is rigid. It is braided tensor equivalent to
  the full subcategory $\CE\subset S\Perv_{\GL(N-1,\bO),q}(\bCD)$ containing all the irreducibles
  $\IC^q_{\bmu,\bnu}$.
  \end{proof}

\subsection{Projective sheaves}
The following proposition will be proved in~\S\ref{proof no stalk} below.

\begin{prop}
  \label{no stalk}
  Let $(\bmu,\bnu)$ be a relevant typical bisignature. Let $(\blambda,\btheta)\ne(\bmu,\bnu)$ be
  a relevant bisignature. Then the costalk of $\IC^q_{\blambda,\btheta}$ at $\BO_{\bmu,\bnu}$ is zero.
\end{prop}

\begin{cor}
  \label{bzeta}
  Let $\bzeta=(N-1,N-2,\ldots,2,1),\ \brho=(N-1,N-2,\ldots,2,1,0)$. Then $\IC^q_{\bzeta,\brho}$ is a
  projective and injective object of $S\Perv_{\GL(N-1,\bO),q}(\bCD)$.
  \end{cor}

\begin{proof}
  An orbit $\BO_{\bmu,\bnu}$ lies in the closure of $\BO_{\bzeta,\brho}$ iff
  $(\bmu,\bnu)\leq(\bzeta,\brho)$,
  that is $(\bzeta,\brho)-(\bmu,\bnu)$ is a nonnegative linear combination of simple roots
  $\alpha_1,\ldots,\alpha_{2N-2}$. This is proved in~\cite[Theorem 3.3.5(a,b)]{bft} for
  $\on{SO}(N-1,\bO)$-orbits in $\Gr_{\on{SO}_N}$. The proof for $\GL_N$ in place of $\on{SO}_N$ is
  absolutely similar.

  One can check that there are {\em no} relevant bisignatures $(\bmu,\bnu)\leq(\bzeta,\brho)$.
  Hence $\IC^q_{\bzeta,\brho}$ is a clean (shriek or star) extension from the orbit $\BO_{\bzeta,\brho}$.
  Hence $\Ext^1_{S\Perv_{\GL(N-1,\bO),q}(\overset{\bullet}\CalD)}(\IC^q_{\bzeta,\brho},\IC^q_{\blambda,\btheta})=
  \Ext^1_{SD^b_{\GL(N-1,\bO),q}(\overset{\bullet}\CalD)}(\IC^q_{\bzeta,\brho},\IC^q_{\blambda,\btheta})$ equals the first
  cohomology (in perverse normalization) of the costalk of $\IC^q_{\blambda,\btheta}$ at $\BO_{\bzeta,\brho}$.
  The latter costalk vanishes by~Proposition~\ref{no stalk}.

  Since $\Ext^1$ from $\IC^q_{\bzeta,\brho}$ to any irreducible object of $S\Perv_{\GL(N-1,\bO),q}(\bCD)$
  vanishes, a standard induction argument shows that $\IC^q_{\bzeta,\brho}$ is a projective object
  of $S\Perv_{\GL(N-1,\bO),q}(\bCD)$. Now applying the Verdier duality (and swapping $q$ and $q^{-1}$)
  we conclude that $\IC^q_{\bzeta,\brho}$ is an injective object of $S\Perv_{\GL(N-1,\bO),q}(\bCD)$.
\end{proof}

\begin{cor}
  \label{proj}
  \textup{(a)} For an arbitrary relevant bisignature $(\bmu,\bnu)$ the fusion product
  $\IC^q_{\bmu,\bnu}\star\IC^q_{\bzeta,\brho}$ is a projective and injective object of
  $S\Perv_{\GL(N-1,\bO),q}(\bCD)$.

  \textup{(b)} Any irreducible sheaf $\IC^q_{\blambda,\btheta}\in S\Perv_{\GL(N-1,\bO),q}(\bCD)$ is
  a quotient of an appropriate projective-injective object in \textup{(a)} above.
\end{cor}

\begin{proof}
  (a) By the rigidity property of $\IC^q_{\bmu,\bnu}$ proved in~Corollary~\ref{rig},
  $\Hom_{S\Perv_{\GL(N-1,\bO),q}(\overset{\bullet}\CalD)}(\IC^q_{\bmu,\bnu}\star\IC^q_{\bzeta,\brho},?)=
  \Hom_{S\Perv_{\GL(N-1,\bO),q}(\overset{\bullet}\CalD)}(\IC^q_{\bzeta,\brho},(\IC^q_{\bmu,\bnu})^*\star?)$.
  The latter functor is exact in the argument~?\ since the fusion product $\star$ is biexact, and
  $\IC^q_{\bzeta,\brho}$ is projective. The projectivity of $\IC^q_{\bmu,\bnu}\star\IC^q_{\bzeta,\brho}$
  follows, and the injectivity is proved the same way.

  (b) Again by rigidity,
  \begin{multline*}
    \Hom_{S\Perv_{\GL(N-1,\bO),q}(\overset{\bullet}\CalD)}(\IC^q_{\bmu,\bnu}
  \star\IC^q_{\bzeta,\brho},\IC^q_{\blambda,\btheta})\\
  =\Hom_{S\Perv_{\GL(N-1,\bO),q}(\overset{\bullet}\CalD)}(\IC^q_{\bmu,\bnu},
  (\IC^q_{\bzeta,\brho})^*\star\IC^q_{\blambda,\btheta}).
  \end{multline*}
  We take an irreducible subsheaf of
  $(\IC^q_{\bzeta,\brho})^*\star\IC^q_{\blambda,\btheta}$ for $\IC^q_{\bmu,\bnu}$.
  \end{proof}

\subsection{Equivalences}

\begin{thm}
  \label{derived}
  The natural functor \[D^b(S\Perv_{\GL(N-1,\bO),q}(\bCD))\to SD^b_{\GL(N-1,\bO),q}(\bCD)\] from the
  derived category of the abelian category $S\Perv_{\GL(N-1,\bO),q}(\bCD)$ to the equivariant
  derived category $SD^b_{\GL(N-1,\bO),q}(\bCD)$ is an equivalence.
\end{thm}

\begin{proof}
  We have to check that for the generating set of objects $\IC^q_{\bmu,\bnu}\star\IC^q_{\bzeta,\brho}$
  we have an isomorphism
  \begin{multline}
    \label{long iso}
    \Hom_{D^b(S\Perv_{\GL(N-1,\bO),q}(\overset{\bullet}\CalD))}(\IC^q_{\bmu,\bnu}
  \star\IC^q_{\bzeta,\brho},\IC^q_{\bmu',\bnu'}\star\IC^q_{\bzeta,\brho})\\
\iso\Hom_{SD^b_{\GL(N-1,\bO),q}(\overset{\bullet}\CalD)}(\IC^q_{\bmu,\bnu}
\star\IC^q_{\bzeta,\brho},\IC^q_{\bmu',\bnu'}\star\IC^q_{\bzeta,\brho})
  \end{multline}
  for any relevant bisignatures $(\bmu,\bnu),\ (\bmu',\bnu')$.
  The LHS of~\eqref{long iso} is Hom between projective-injective objects of the abelian category
  $S\Perv_{\GL(N-1,\bO),q}(\bCD)$, so it is concentrated in degree~$0$ and by rigidity equals
  \[\Hom_{S\Perv_{\GL(N-1,\bO),q}(\overset{\bullet}\CalD)}(\IC^q_{\bzeta,\brho},(\IC^q_{\bmu,\bnu})^*\star
  \IC^q_{\bmu',\bnu'}\star\IC^q_{\bzeta,\brho}),\] i.e.\ the multiplicity of $\IC^q_{\bzeta,\brho}$ in
  $(\IC^q_{\bmu,\bnu})^*\star\IC^q_{\bmu',\bnu'}\star\IC^q_{\bzeta,\brho}$.

  Again by rigidity, the RHS of~\eqref{long iso} equals
  \[\Hom_{SD^b_{\GL(N-1,\bO),q}(\overset{\bullet}\CalD)}(\IC^q_{\bzeta,\brho},(\IC^q_{\bmu,\bnu})^*\star
  \IC^q_{\bmu',\bnu'}\star\IC^q_{\bzeta,\brho}).\]
  Since $\IC^q_{\bzeta,\brho}$ is a clean extension from $\BO_{\bzeta,\brho}$, the latter Hom equals
  the costalk of $(\IC^q_{\bmu,\bnu})^*\star\IC^q_{\bmu',\bnu'}\star\IC^q_{\bzeta,\brho}$ at $\BO_{\bzeta,\brho}$.
  By~Proposition~\ref{no stalk}, the latter costalk is nontrivial only in degree~0 (in perverse
  normalization), and equals the multiplicity of $\IC^q_{\bzeta,\brho}$ in
  $(\IC^q_{\bmu,\bnu})^*\star\IC^q_{\bmu',\bnu'}\star\IC^q_{\bzeta,\brho}$.

  The theorem is proved.
\end{proof}

\begin{thm}
  \label{main}
  The functor $F\colon S\Perv_{\GL(N-1,\bO),q}(\bCD)\to\on{FS}$ is a braided tensor equivalence.
\end{thm}

\begin{proof}
  In view of~Lemma~\ref{etingof}, it remains to prove that the embedding of the full subcategory
  $\CE\hookrightarrow S\Perv_{\GL(N-1,\bO),q}(\bCD)$ is essentially surjective. But $\CE$ contains
  all the irreducibles $\IC^q_{\bmu,\bnu}$ and is closed under fusion products, so it contains
  all the projective-injective generators $\IC^q_{\bmu,\bnu}\star\IC^q_{\bzeta,\brho}$
  of~Corollary~\ref{proj}. Every object of $S\Perv_{\GL(N-1,\bO),q}(\bCD)$ is the image of a
  morphism from a direct sum of projective objects of the above type to a direct sum of injective
  objects of the above type, hence it is contained in $\CE$.

  This finishes the proof of our main theorem modulo~Proposition~\ref{no stalk} that will be dealt
  with in the next Section.
\end{proof}

Now we can compose the derived equivalence $D^b\Rep_q(\GL(N-1|N))\iso D^b(\on{FS})$
of~Theorem~\ref{bfsl} with the quasiinverse of the derived equivalence
$D^b(S\Perv_{\GL(N-1,\bO),q}(\bCD))\iso D^b(\on{FS})$ of~Theorem~\ref{main} and the equivalence
$D^b(S\Perv_{\GL(N-1,\bO),q}(\bCD))\iso SD^b_{\GL(N-1,\bO),q}(\bCD)$ of~Theorem~\ref{derived}
to obtain

\begin{cor}
  The above composition of equivalences gives rise to a braided tensor equivalence
  \[D^b\Rep_q(\GL(N-1|N))\iso SD^b_{\GL(N-1,\bO),q}(\bCD).\]
\end{cor}

\section{The case $C=\BA^1$}
\label{slices}
The goal of this section is a proof of~Proposition~\ref{no stalk}. To this end we restrict
our considerations to a particular curve $C=\BA^1$. Since $C$ is assumed to be projective this
means that actually $C=\BP^1$, but in all the relevant moduli spaces we change the base to the
open subspace of configurations on $\BP^1$ avoiding the point $\infty\in\BP^1$. We also set the
marked point $c=0\in\BP^1$. We keep all the
notation of the previous sections, but apply it in the above sense.

\subsection{Thick affine Grassmannian}
Let $\GR_{\GL_N}$ denote the Kashiwara scheme (of infinite type, alias thick affine Grassmannian)
solving the moduli problem of the following data:

(a) A vector bundle $\CU$ of rank $N$ on $\BP^1$;

(b) A trivialization $\tau_N$ of $\CU$ in the formal neighbourhood $\wh\BP{}^1_\infty$.

\noindent We will also use the thick affine Grassmannian $\GR_{\GL_{N-1}}$ for $\GL_{N-1}$ and
$\GR_G=\GR_{\GL_{N-1}}\times\GR_{\GL_N}$ for $G=\GL_{N-1}\times\GL_N$. Namely, we will construct a
morphism $\bs\colon\CW_{\leq(\blambda,\btheta)}^{(\bmu,\bnu)}\to\GR_G$. To this end note that
since under our standing assumption $C=\BA^1$ the $B^-_{N-1}$- and $B^-_N$-structures
of~\S\ref{convo}e,f) are in general position at $\infty\in\BP^1$, both $\CV$ and $\CU$ are
trivialized at $\infty\in\BP^1$. Moreover, the $B^-_{N-1}$- and $B^-_N$-structures are in general
position in a Zariski neighbourhood of $\infty\in\BP^1$, so both $\CV$ and $\CU$ are trivialized
in this neighbourhood, and a fortiori in the formal neighbourhood $\wh\BP{}^1_\infty$. These are
the desired trivializations $\tau_{N-1},\tau_N$.

The ratio $\tau_N\circ(\tau_{N-1}\oplus1)^{-1}$
coincides with $\sigma$ of~\S\ref{convo}d) restricted from $\BP^1\setminus\{0\}$ to $\wh\BP{}^1_\infty$.
Here $\tau_{N-1}\oplus1$ denotes the induced trivialization of $\CV\oplus\CO_{\BP^1}$.

\begin{lem}
  \label{s}
  Assume that $(\bmu,\bnu)$ is a bisignature. Then the morphism $\bs$ is universally injective
  (radicial):
  $\CW_{\leq(\blambda,\btheta)}^{(\bmu,\bnu)}\hookrightarrow\GR_{\GL_{N-1}}\times\GR_{\GL_N}$.
\end{lem}

\begin{proof}
  The argument is the same as the one in the proof of~\cite[Theorem 2.8]{bf}. Let us recall it
  in the case when both $\bmu$ and $\bnu$ are regular (that is $\mu_i>\mu_{i+1}$ and $\nu_i>\nu_{i+1}$).
  The general case is similar, but requires a more cumbersome notation.
  Note that~Proposition~\ref{no stalk} is only used for $(\bmu,\bnu)=(\bzeta,\brho)$
  of~Corollary~\ref{bzeta}, and $\bzeta,\brho$ are both regular.

  With respect to the trivialization at $\infty\in\BP^1$, the complete flags
  in $\CV|_\infty$ and $\CU|_\infty$ take the values specified in the beginning of~\S\ref{Factor zas}.

  Now the dominance assumption on the degrees $\bmu,\bnu$ of the complete flags in $\CV,\CU$
  guarantees that the isomorphism type of the vector bundle $\CV$ (resp.\ $\CU$) on $\BP^1$
  is $\bmu$ (resp.\ $\bnu$), and the complete flags are nothing but the Harder-Narasimhan flags.
  (If $\bmu,\bnu$ are not necessarily regular, the Harder-Narasimhan flags are not necessarily
  complete, but they can be uniquely refined to complete flags with the value at $\infty\in\BP^1$
  prescribed by the previous paragraph.)

  Thus all the data of ~\S\ref{convo}a--f) are uniquely reconstructed from $\CV,\CU,\tau_{N-1},\tau_N$
  (recall that $\sigma=\tau_N\circ(\tau_{N-1}\oplus1)^{-1}$).
\end{proof}

\subsection{Contraction}
In case $(\bmu,\bnu)$ is a bisignature, we define the following point
$w\in\CW_{\leq(\blambda,\btheta)}^{(\bmu,\bnu)}$. We set 
\begin{gather*}
\CV=\CO_{\BP^1}(\mu_{N-1}\cdot0)e_{N-1}\oplus\ldots\oplus\CO_{\BP^1}(\mu_1\cdot0)e_1,\\
\CU=\CO_{\BP^1}(-\nu_1\cdot0)(e_1+e_N)\oplus\ldots\oplus\CO_{\BP^1}(-\nu_{N-1}\cdot0)(e_{N-1}+e_N)
\oplus\CO_{\BP^1}(-\nu_N\cdot0)e_N.
\end{gather*}
The identification $\sigma$ of $(\CV\oplus\CO_{\BP^1}e_N)|_{\BP^1\setminus\{0\}}$ and $\CU|_{\BP^1\setminus\{0\}}$
is tautological. The complete flag in $\CV$ (resp.\ $\CU$) is formed by the first line bundle,
the direct sum of the first and second line bundles, and so on. The colored divisor $D$ is supported
at $0\in\BP^1$.

The group $\BG_m$ of loop rotations acts naturally on $\CW_{\leq(\blambda,\btheta)}^{(\bmu,\bnu)}$.

\begin{lem}
  \label{contract}
  Assume that $(\bmu,\bnu)$ is a bisignature. Then the action of loop rotations contracts
  $\CW_{\leq(\blambda,\btheta)}^{(\bmu,\bnu)}$ to the point $w$.
\end{lem}

\begin{proof}
  Recall that $\GR_G=\bigsqcup_{(\bxi,\bta)\in X^+}\GR_G^{\bxi,\bta}$ is stratified according to the isomorphism
  types of $G$-bundles on $\BP^1$. Each stratum is contracted by the loop rotations to a finite
  dimensional subscheme isomorphic to a partial flag variety of $G$. The isomorphism takes the value
  of the Harder-Narasimhan flag at $\infty\in\BP^1$.

  From the proof of~Lemma~\ref{s}, the image $\bs(\CW_{\leq(\blambda,\btheta)}^{(\bmu,\bnu)})\subset\GR_G$
  lies in the stratum $\GR_G^{\bmu,\bnu}$. More precisely, it lies in the fiber of the contraction
  morphism over the point $(B_{N-1}^-,B_N^-)$ in the flag variety of $G$ (or rather over its image
  in the relevant partial flag variety). But $\bs$ is clearly equivariant with respect to the loop
  rotations, and $\bs(w)$ is the fixed point of the contraction morphism.
\end{proof}

\begin{rem}
  Thus $\CW_{\leq(\blambda,\btheta)}^{(\bmu,\bnu)}$ plays the role of a transversal slice to the orbit
  $\sO_{\bmu,\bnu}\subset\ol\sO_{\blambda,\btheta}\subset\Gr_G\subset\GR_G$ in the case when $(\bmu,\bnu)$
  is a bisignature. Note that the definition of $\CW_{\leq(\blambda,\btheta)}^{(\bmu,\bnu)}$ strongly resembles
  the symmetric definition of (generalized) transversal slices in~\cite[\S2(v)]{bfn}.
\end{rem}

\subsection{Proof of Proposition~\ref{no stalk}}
\label{proof no stalk}
The configuration space $\BA^{(\bmu,\bnu)}_{\leq(\blambda,\btheta)}$ is just an affine space, so
the line bundle $\CP$ can be trivialized. We fix a trivialization (note that it is unique up to a
scalar multiplication), that is the corresponding nowhere vanishing section. In the argument below,
we will consider the restrictions of factorizable sheaves to this section, but we will keep the old
notation in order not to make it more cumbersome.

According to~Lemma~\ref{Two line}, the line bundle $\jmath^*\bp^*\CalD$ on
$\CW_{\leq(\blambda,\btheta)}^{(\bmu,\bnu)}$ trivializes as well (i.e.\ acquires a nowhere vanishing section),
and we will consider the restriction of $\jmath^*\bp^\circ\IC^q_{\blambda,\btheta}$ to this section, but
again we will keep notation $\jmath^*\bp^\circ\IC^q_{\blambda,\btheta}$ for this restriction below.
In particular, we have to compute the costalk of $\jmath^*\bp^\circ\IC^q_{\blambda,\btheta}$ at
(the section over) the point $w$.

By the contraction principle and~Lemma~\ref{contract}, the above costalk equals
$H^\bullet_c(\CW_{\leq(\blambda,\btheta)}^{(\bmu,\bnu)},\jmath^*\bp^\circ\IC^q_{\blambda,\btheta})$.
By the cleanness property~Lemma~\ref{clean} and the irreducibility result of~Lemma~\ref{irreduc},
the latter cohomology equals
$H^\bullet_c(\BA^{(\bmu,\bnu)}_{\leq(\blambda,\btheta)},\CF_{\blambda,\btheta}^{(\bmu,\bnu)})$.
The configuration space $\BA^{(\bmu,\bnu)}_{\leq(\blambda,\btheta)}$ is just an affine space contracted to
the origin~$0$ by the loop rotations. Invoking the contraction principle once again, we conclude
that the latter cohomology equals the costalk of $\CF_{\blambda,\btheta}^{(\bmu,\bnu)}$ at
$0\in\BA^{(\bmu,\bnu)}_{\leq(\blambda,\btheta)}$. Under the equivalence of~Theorem~\ref{bfsl}, the latter
costalk is isomorphic to $\Ext^\bullet(M_{\bmu,\bnu},V_{\blambda,\btheta})$. Here $M_{\bmu,\bnu}$ stands
for the Verma module with highest weight $(\bmu,\bnu)$ over $U_q(\fgl(N-1|N))$, and the Ext is
taken in the category $\CO$. Finally, the latter Ext vanishes since $M_{\bmu,\bnu}$ and
$V_{\blambda,\btheta}$ lie in the different linkage classes of the category $\CO$ (they are separated
by the eigenvalues of the center of $U_q(\fgl(N-1|N))$).

The proposition is proved. \hfill $\Box$

\end{document}